\def\rgbcolor#1{\pdfliteral{#1 rg #1 RG}}
\def\lightgray{\rgbcolor{0.7 0.7 0.7}}
\def\black{\rgbcolor{0 0 0}}
\def\ghost#1{{\lightgray #1 \black}}
\documentclass[11pt]{amsart}

\usepackage{amsmath, amsthm, amsfonts, amssymb, latexsym, graphicx}
\usepackage{fullpage}
\usepackage[bookmarks,colorlinks,linkcolor=blue,
	    citecolor=blue,urlcolor=blue]{hyperref}

\swapnumbers
\newtheorem{theorem}{Theorem}
\newtheorem{lemma}[theorem]{Lemma}
\newtheorem{proposition}[theorem]{Proposition}
\newtheorem{corollary}[theorem]{Corollary}

\theoremstyle{definition}
\newtheorem*{definition}{Definition}

\theoremstyle{remark}
\newtheorem*{remark}{Remark}

\numberwithin{equation}{section}
\numberwithin{theorem}{section}

\begin{document}
\title{Rotation Remainders}
\author{P. Jameson Graber, Washington and Lee University '08}
\maketitle

\begin{abstract}
We study properties of an array of numbers, called ``the triangle," in which each row is formed by rotating all the numbers in the previous row to the left by $m$ positions in cyclical fashion, then appending a number to the end of the row.
We show that a number's position in the triangle is uniquely determined by the infinite sequence of column positions--called ``rotation remainders"--which we track as the number repeatedly rotates back to the first $m$ columns.
The rotation remainders can be viewed as the digits in a ``base $m/(m+1)$" expansion in an ``$m$-adic" topological ring of a number encoding a given position in the triangle.
Properties of these expansions are used to prove interesting claims about the triangle, such as the aperiodicity of any sequence of rotation remainders.

This article was the author's senior thesis toward the completion of a mathematics major at Washington and Lee University.
\end{abstract}

\section{Introduction}

An object in mathematics that first appears for the sake of amusement can demonstrate deep connections with well-known questions about numbers.  Consider an array of numbers formed by a rotating queue: starting with just the number 1, to obtain the next row we move everything in the last row $m$ steps to the left, with numbers at the front of the row cycling around and appearing at the back.  We then append 1 plus the head of the last row to the new row.  Here is an example showing the first few rows when $m = 3$.

\[\begin{array}{cccccccccc}
1 & & & & & & & & &\\
1 & 2 & & & & & & & & \\
2 & 1 & 2 & & & & & & &\\
2 & 1 & 2 & 3 & & & & & &\\
3 & 2 & 1 & 2 & 3 & & & & &\\
2 & 3 & 3 & 2 & 1 & 4 & & & &\\
2 & 1 & 4 & 2 & 3 & 3 & 3 & & &\\
2 & 3 & 3 & 3 & 2 & 1 & 4 & 3 & & \\
3 & 2 & 1 & 4 & 3 & 2 & 3 & 3 & 3 &\\
4 & 3 & 2 & 3 & 3 & 3 & 3 & 2 & 1 & 4\\
\end{array}\]

Note that each individual number can be tracked as it goes through the rotations; the 2 at the head of row 5 is different from the 2 at the head of row 6.  We can imagine this situation physically as people standing in line (or, better yet, a circle) with one more person being added to the line at each turn.  Before a person is added to the line, everyone in the line rotates according to the scheme described above--everyone moves up by $m$ spaces, with those at the front being moved to the back.  For some of the central questions in this paper, it may be more helpful just to view the situation in this way, and forget about the number each person is carrying.

Each number will repeatedly rotate back to one of the first $m$ columns.  Suppose we label the columns $0,1,2,\ldots$ and we want to track the following behaviors:

\begin{enumerate}

\item the column positions in which a number appears as it repeatedly moves back to the first $m$ columns,

\item the frequency with which 1 is at the head of a row,

\item the frequency with which a new number appears in the array, and

\item the frequency with which a certain number appears in a row.

\end{enumerate}

This paper will explore these questions, with the largest amount of work being spent on the first question.  In this introduction, we will introduce some notation to get us started.  Then we will give our main results, which will be proved in the main body of the paper.  The introduction will conclude with an outline of the main body of the paper.

Here we introduce the main character of this paper.

\begin{definition} The symbol $T_m$ will refer to the array described above.  Occasionally we refer to $T_m$ as \emph{the triangle}. \end{definition}

\begin{definition} Let $x$ be a positive integer.  Then let $r$ be any integer, and let $r_0$ be $r$ reduced modulo $x$.  The symbol $T_m(x,r)$ will refer to the number that appears in row $x$ and column $r_0$ of the triangle $T_m$.  Note the important distinction between the \emph{triangle} $T_m$ and the \emph{number} $T_m(x,r)$. \end{definition}

\begin{definition} The positive integer $m$ will be referred to as the \emph{rotation number}. \end{definition}

We can write out the contents of $T_m$ exactly, using the following recursive relations.

\begin{itemize}
\item $T_m(1,0) = 1$, 
\item $T_m(x,r) = T_m(x-1,r+m),$ and
\item $T_m(x,x-1) = 1 + T_m(x-1,0),$
\end{itemize}
the last two conditions holding for all $x > 1$ and $0 \leq r \leq x - 2$.
Note that when we say $T_m(x,r) = T_m(x-1,r+m)$, we are secretly doing a rotation, but the ``wrap-around" is built into the definition of $T_m(x,r)$ since $r$ is reduced modulo $x$ to get the actual column position of $T_m(x,r)$.  Here is an example to show what we are doing.  We know from the opening table that the following two rows appear in $T_3$:

\[\begin{array}{cccccccccc}
3 & 2 & 1 & 4 & 3 & 2 & 3 & 3 & 3 &\\
4 & 3 & 2 & 3 & 3 & 3 & 3 & 2 & 1 & 4\\
\end{array}\]

The 4 in column 3 of row 9 moves to column 0 of row 10, the 3 in column 4 of row 9 moves to column 1 or row 10, and so on; i.e. $T_3(10,0) = T_3(9,3), T_3(10,1) = T_3(9,4),$ etc.  But when we say $T_3(10,7) = T_3(9,10)$, we are not saying the 2 in column 7 of row 10 came from column 10 of row 9, since there is nothing in column 10 of row 9.  By definition, however, $T_3(9,10) = T_3(9,1)$, since 10 reduced modulo 9 is 1.  Thus we are saying the 2 in column 7 of row 10 came from column 1 of row 9, which is exactly what happens in the rotation.  This way of defining $T_m(x,r)$ for $r \geq x$ allows us to think of each row as a cycle, which is what we will need in order to correctly think about the questions posed in this paper.

Motivated by the first question stated above, let us introduce one last notation.

\begin{definition} Let $x$ be a positive integer and let $r$ be a least nonnegative residue mod $m$.  Suppose $T_m(x,r) = N$.  We define the sequences $\{x_n(x,r)\}_{n=0}^\infty$ and $\{r_n(x,r)\}_{n=0}^\infty$ as follows.  Let $x_0(x,r) = x$ and $r_0(x,r) = r$.  Then for each positive integer $n$, let $x_n(x,r)$ and $r_n(x,r)$ be the smallest possible integers satisfying the following:

\begin{enumerate}

\item $x_0(x,r) \leq x_1(x,r) \leq \cdots \leq x_n(x,r)$;

\item $0 \leq r_n(x,r) \leq m - 1$;

\item if $x_n(x,r) = x_{n-1}(x,r)$, then $r_n(x,r) > r_{n-1}(x,r)$; and

\item $T_m(x_n(x,r),r_n(x,r)) = N$.

\end{enumerate}

The sequences $\{x_n(x,r)\}_{n=0}^\infty$ and $\{r_n(x,r)\}_{n=0}^\infty$ are said to track the number $N$ down the first $m$ columns of $T_m$. \end{definition}

In terms of people standing in a line, $\{x_n\}$ and $\{r_n\}$ would be defined as follows.  Suppose a person standing in the line wrote down his position every time he ended up in the first $m$ spots in line.  He would write down how many people are in the row (this would be $x_n$) and how far he is from the front of the line (this would be $r_n$).  To give an example, we go back to the $T_3$ and track the ``1," while filling in some empty spots:

\def\A{\fbox{\textbf 1}}

\[\begin{array}{cccccccccc}
0 & 1 & 2 \\ 
\cline{1-3} \\ [-2ex]
1 & \ghost{\A} & \ghost{\A} & & & & & & &\\
\A & 2 & \ghost{\A} & & & & & & & \\
2 & \A & 2 & & & & & & &\\
2 & \A & 2 & 3 & & & & & &\\
3 & 2 & \A & 2 & 3 & & & & &\\
2 & 3 & 3 & 2 & 1 & 4 & & & &\\
2 & \A & 4 & 2 & 3 & 3 & 3 & & &\\
2 & 3 & 3 & 3 & 2 & 1 & 4 & 3 & & \\
3 & 2 & \A & 4 & 3 & 2 & 3 & 3 & 3 &\\
4 & 3 & 2 & 3 & 3 & 3 & 3 & 2 & 1 & 4\\
\end{array}\]

The faded 1's have been added to the triangle legitimately, because $T_3(1,2) = T_3(1,1) = T_3(1,0)$ and $T_3(2,2) = T_3(2,0)$ by definition of $T(x,r)$.  By reading off the position of the 1 in the first 3 columns, we can see that

\begin{align}
\label{1 sequence} \{x_n(1,0)\} &= \{1,1,1,2,2,3,4,5,7,9,\ldots\}, \text{ and}\\
 \{r_n(1,0)\} &= \{0,1,2,0,2,1,1,2,1,2,\ldots\}. \notag
\end{align}

There is an interesting way to encode the sequences $\{x_n(x,r)\}$ and $\{r_n(x,r)\}$.  Let $y_0(x,r) = (m+1)x_0(x,r) + r_0(x,r) = (m+1)x + r$.  For each $n \in \mathbb{N}$, let $y_n(x,r) = \left\lfloor \frac{(m+1)y_{n-1}(x,r)}{m} \right\rfloor$.  The following result shows why this is such a cleverly designed sequence.

\begin{proposition} \label{y_n} For each non-negative integer $n$, $\left\lfloor \frac{y_n(x,r)}{m+1} \right\rfloor = x_n(x,r)$ and $y_n(x,r) \equiv r_n(x,r) \bmod{(m+1)}$.  Stated more colloquially, $x_n(x,r)$ is the quotient and $r_n(x,r)$ is the remainder when $y_n(x,r)$ is divided by $m + 1$. \end{proposition}

Now that we have introduced all the necessary notation concerning the elements of $T_m$, we are ready to give our main results.  The first theorem requires a definition.

\begin{definition} \label{D_m} Let $D_m$ be the set of all rational numbers that can be written in the form $a/b$, where $a$ and $b$ are integers that are relatively prime and $b$ is relatively prime to $m$. \end{definition}

In the next section, we will prove that $D_m$ is a topological ring.  We will then make our way to a proof of the following theorem.

\begin{theorem} \label{m-adic sum} There exists a topological ring $\mathbb{Q}_m$ with the following properties:
\begin{enumerate}

\item $\mathbb{Q}_m$ is a complete metric space.

\item $\mathbb{Q}_m$ contains a sub-ring isomorphic to $D_m$.

\item In $\mathbb{Q}_m$, for any starting row $x$ and column $r$, 
\begin{equation} \label{main expansion}
m \cdot y_0(x,r) = \sum_{k=1}^\infty \left(\frac{m}{m+1}\right)^k r_k(x,r).
\end{equation}
\end{enumerate}
\end{theorem}

\begin{remark} \label{m-adic sum remark} As a consequence of part (3) in Theorem \ref{m-adic sum}, the sequence $\{r_n(x,r)\}_{n=1}^\infty$ determines $y_0(x,r)$ and hence determines $x$ and $r$.  We might even be tempted to think of the sequence $\{r_n(x,r)\}$ as the ``base $\frac{m}{m+1}$" expansion of $y_0(x,r)$.  We will show later how such expansions can be treated as elements in a ring independent of their corresponding elements in $D_m$. \end{remark}

\begin{remark} \label{m-adic sum remark 2} Although $1/m$ is not an element of $\mathbb{Q}_m$, it is still reasonable to write \[y_0(x,r) = \dfrac{1}{m} \sum_{k=1}^\infty \left(\frac{m}{m+1}\right)^k r_k(x,r),\] because all the terms in the sum on the right have at least one factor of $m$ in their numerators. \end{remark}

Theorem \ref{m-adic sum} tells us that in the triangle $T_m$, the column positions of a number as it appears within the first $m$ columns can be used to determine the position in which the number started.  This is somewhat surprising--although it seems natural to determine the column positions from the starting position, it is difficult to imagine doing the opposite.  Nevertheless, this special sequence $\{r_n(x,r)\}$ can be used to determine $x$ and $r$ using this bizarre sum.  In terms of people standing in a line, it means that no two people can appear to have parallel paths.  In other words, if you are standing in the line, writing down how far you are from the front each time you make it to the first $m$ spaces in the line, the sequence you get is unique to you.  In fact, every \emph{tail} of this sequence is unique to you, as the following corollary implies.

\begin{corollary} \label{tail story} Suppose we have $\bar{x}$ and $\bar{r}$ such that $\{r_n(\bar{x},\bar{r})\}_{n=0}^\infty$ is a tail of $\{r_n(x,r)\}_{n=0}^\infty$.  That is, suppose there exists some positive integer $j$ such that $r_n(\bar{x},\bar{r}) = r_{n+j}(x,r)$ for every $n$.  Then $\bar{x} = x_j(x,r)$ and $\bar{r} = r_j(x,r)$. \end{corollary}

Another consequence is that the tail of a sequence of $\{r_n(x,r)\}_{n=0}^\infty$ determines the beginning of the sequence, no matter where we choose to start the tail.

\begin{corollary} \label{equal tails} Suppose $x, r, \bar{x}$ and $\bar{r}$ are integers such that $\{r_n(x,r)\}_{n=j}^\infty = \{r_n(\bar{x},\bar{r})\}_{n=j}^\infty$ for some positive integer $j$.  Then $\{r_n(x,r)\}_{n=0}^\infty = \{r_n(\bar{x},\bar{r})\}_{n=0}^\infty$. \end{corollary}

An intriguing fact about $\{r_n(x,r)\}$ is that it turns out not to have any repeating patterns in it.  Suppose again you are standing in the line, writing down how far you are from the front each time you make it to the first $m$ spaces in the line.  Then you notice that you have written the same few numbers several times over in the same order.  You might wonder whether this will go on forever.  The following theorem assures you that it will not.

\begin{theorem} \label{aperiodic} For any starting row $x$ and column $r$, the sequence of column positions $\{r_n(x,r)\}$ is aperiodic. \end{theorem}

The final result in our introduction focuses on a particular triangle, $T_2$.  The general question, whether or not every person in the line will eventually come to the front of the line, is a very difficult question for $m > 2$.  However, the answer is fairly easy for $m = 2$, and we will spend an entire section later on finding out as much as we can about what happens in this case.

\begin{theorem} \label{m equals 2} A given number in $T_2$ will appear at the head of infinitely many rows.
\end{theorem}

\begin{remark} \label{m equals 2 remark} Theorem \ref{m equals 2} does not tell us which rows a particular number will lead.  We will discuss this in a later section. \end{remark}

These theorems constitute our main results.  To conclude this introduction, we give an outline of the main body of the paper.  We will spend the largest amount of effort on question (1) stated above, that is, on the sequences of row and column positions of an element in $T_m$.  This effort will require three sections of technical background, in which we discuss the rings $D_m$ and $\mathbb{Q}_m$ and introduce the concept of a \emph{rotation remainder expansion} of a number.  Then we will have three sections of proofs for Theorems \ref{m-adic sum},  \ref{aperiodic}, and \ref{equal tails}.  Next, we will explore the remaining questions posed in this introduction by limiting our view to the cases $m = 1$ and $m = 2$.  Finally, we will demonstrate the connection between our triangle $T_m$ and the problem of Josephus, a famous problem in number theory.

{\bf Acknowledgement.} Special thanks to Jacob Siehler for directing this senior thesis project, and for prompting the author to publish the results.

\section{The Ring $D_m$}

In the introduction, we defined $D_m$ as $\{a/b \in \mathbb{Q} : \gcd(m,b) = 1\}$.  We promised to show that $D_m$ is a topological ring.  This will be the primary purpose of this section.  Additionally, we will demonstrate that $D_m$ is not complete, thereby showing the need to introduce $\mathbb{Q}_m$ in the next section.

We now establish that $D_m$ is a ring.  In particular, it is a subring of the rational numbers.  All the properties of addition and multiplication in $D_m$ are inherited from the rationals.  Thus it suffices to show that $D_m$ is closed under these operations.

Let $a/b$ and $c/d$ be in $D_m$ so that $b$ and $d$ are each relatively prime to $m$.  Since $bd$ is also relatively prime to $m$, it follows that $\frac{a}{b} + \frac{c}{d} = \frac{ad + bc}{bd}$ and $\frac{a}{b}\frac{c}{d} = \frac{ac}{bd}$ are each in $D_m$.  Hence $D_m$ is closed under ring addition and multiplication.  Closure under additive inverses is trivial to show.  Therefore $D_m$ is a subring of the rationals.

Next, we define a norm and a metric on $D_m$.

\begin{definition} Let $a/b \in D_m$ be given.  We define $\left| \frac{a}{b} \right|_m$ to be $m^{-k}$, where $k$ is the greatest integer such that $m^k$ divides $a$.  (For $a = 0$ we say that $k = \infty$ and that $|0|_m = 0$.)  We call $| \cdot |_m$ the ``$m$-adic" norm on $D_m$. \end{definition}

Before we define a metric based on this norm, two remarks are in order.  First, the $m$-adic norm on $D_m$ is very similar to the notion of a ``$p$-adic" norm on the rationals, where $p$ is a prime \cite{gouvea1997p}.  A ``$p$-adic" norm defined on rationals $a/b$ is based on how many times $p$ divides $a$ \emph{and} how many times $p$ divides $b$.  Because $m$ need not be prime, we do not attempt to include any elements in $D_m$ that have denominators not relatively prime to $m$.  Thus, for any $a/b \in D_m$, $\left| \frac{a}{b} \right|_m \leq 1$.  The second remark is an obvious but useful observation: for any $a/b \in D_m$, $|a/b|_m = |a|_m$.  We will use this fact often without reference.

Now that we have established a norm on $D_m$, we proceed to define a metric.

\begin{definition} Let the function $d : D_m \times D_m \rightarrow [0,1]$ be defined in the following way: for any element $a/b \in D_m$, let $d\left(\frac{a}{b},\frac{c}{d}\right) = \left| \frac{a}{b} - \frac{c}{d} \right|_m.$ \end{definition}

\begin{proposition} \label{metric} The function $d$ is a metric on $D_m$. \end{proposition}

\begin{proof} Because the divisors of $a$ and $-a$ are exactly the same, $\left| \frac{a}{b} \right|_m = \left| - \frac{a}{b} \right|_m$, and thus $d$ is symmetric.  Because $|0|_m = 0$, $d\left(\frac{a}{b},\frac{a}{b}\right) = 0$.  Now we must check the triangle inequality.  Let $x = \frac{a}{b}, y = \frac{c}{d},$ and $z = \frac{e}{f}$.  Observe that \begin{equation*}x - y = \frac{ad - bc}{bd}, y - z = \frac{cf - de}{df}, \text{ and } x - z = \frac{af - be}{bf}.\end{equation*}  Let $|a|_m = m^{-k_a}, |c|_m = m^{-k_c}$ and $|e|_m = m^{-k_e}$.  The greatest power of $m$ dividing $ad - bc$ is $\min\{m^{k_a}, m^{k_c}\}$, the greatest power of $m$ dividing $cf - de$ is $\min\{m^{k_c}, m^{k_e}\}$, and the greatest power of $m$ dividing $af - be$ is $\min\{m^{k_a}, m^{k_e}\}$.  It follows that $d(x,y) = \max\{m^{-k_a},m^{-k_c}\}$, $d(y,z) = \max\{m^{-k_c},m^{-k_e}\}$, and $d(x,z) = \max\{m^{-k_a},m^{-k_e}\}$.  From this we get the decisive inequality \begin{equation*}d(x,y) + d(y,z) \geq \max\{m^{-k_a},m^{-k_c},m^{-k_e}\} \geq \max\{m^{-k_a},m^{-k_e}\} = d(x,z).\end{equation*}  Thus the triangle inequality holds, and $d$ is a metric. 
\end{proof}

Now we will show that $D_m$ is a topological ring.  We must show that addition and multiplication are continuous functions from $D_m \times D_m$ to $D_m$.  The proof for addition requires only the triangle inequality.  The proof for multiplication requires the following lemma.

\begin{lemma} \label{mult-lemma} Suppose $x$ and $y$ are in $D_m$.  Then $|xy|_m \leq |x|_m|y|_m \leq |y|_m$. \end{lemma}

\begin{proof} Let $x = \frac{a}{b}$ and let $y = \frac{c}{d}$.  Let $|x|_m = m^{-k}$ and $|y|_m = m^{-l}$.  Then $a = em^k$ and $c = fm^l$ for some integers $e$ and $f$.  Thus $ac = efm^{k + l}$; hence $m^{k+l} \mid ac$.  It follows that $|ac|_m \leq m^{-(k+l)} = |x|_m|y|_m$; hence $|xy|_m \leq |x|_m|y|_m$.  This proves the left-hand inequality.

The right-hand inequality follows, since $|x|_m \leq 1$.  This proves the lemma. \end{proof}

\begin{proposition} \label{D_m topological} The ring $D_m$ is a topological ring. \end{proposition}

\begin{proof} 
First we show that the addition operation is continuous.  Suppose $(x,y)$ is a point in $D_m \times D_m$ and that $\{(x_n,y_n)\}_{n=1}^\infty$ is a sequence converging to $(x,y)$.  Let $\epsilon > 0$ be given.  Then there exists a positive integer $K$ such that whenever $n \geq K$, then $|x_n - x|_m < \epsilon/2$ and $|y_n - y|_m < \epsilon/2$.  Since the triangle inequality holds, whenever $n \geq K$ we have
\begin{align*}
|(x_n + y_n) - (x + y)|_m &\leq |(x_n + y_n) - (x + y_n)|_m + |(x + y_n) - (x + y)|_m\\
&= |x_n - x|_m + |y_n - y|_m < \epsilon/2 + \epsilon/2 < \epsilon.
\end{align*}
Therefore, $x_n + y_n \rightarrow x + y$.  It follows that addition is continuous.

Now using Lemma \ref{mult-lemma}, whenever $n \geq K$ we also have
\begin{align*}
|x_n y_n - xy|_m &\leq |x_n y_n - xy_n|_m + |xy_n - xy|_m\\
&\leq |x_n - x|_m + |y_n - y|_m < \epsilon/2 + \epsilon/2 = \epsilon.
\end{align*}

Therefore, $x_n y_n \rightarrow xy$.  It follows that multiplication is continuous.  Since addition and multiplication are both continuous from $D_m \times D_m$ to $D_m$, it follows that $D_m$ is a topological ring. 
\end{proof}

The next fact we will show is that $D_m$ is not complete.  That is, there must be Cauchy sequences in $D_m$ such that fail to converge.  First, observe that any sequence of the form $\{\sum_{k=1}^n m^k s_k\}_{n=1}^\infty$ in $D_m$ is a Cauchy sequence.  Now the following lemma and its corollary will establish that $D_m$ is not complete.

\begin{lemma} \label{0 series} Let $\{s_k\}_{k=0}^\infty$ be a sequence of least nonnegative residues modulo $m$.  Then the series $\sum_{k=0}^\infty m^k s_k$ converges to 0 in $D_m$ if and only if $s_0 = s_1 = s_2 = \cdots = 0$. \end{lemma}

\begin{proof} Suppose $\sum_{k=0}^\infty m^k s_k = 0$. Suppose, in order to obtain a contradiction,  that there are non-zero elements in $\{s_k\}$.  Let $p$ be the smallest integer such that $s_p \neq 0$.  Since $\sum_{k=0}^\infty m^k s_k = 0$, there exists some $N_1$ such that whenever $n \geq N_1$, $\left| \sum_{k=0}^n m^k s_k \right|_m < m^{-p}$.  Let $N = \min\{p, N_1\}$.  Now \begin{equation*}\left| \sum_{k=0}^N m^k s_k \right|_m = \left|\sum_{k=p}^N m^k s_k \right|_m < m^{-p},\end{equation*} which implies that $m^{p+1}$ must divide $\sum_{k=p}^N m^k s_k$.  Since $m^{p+1}$ divides $\sum_{k=p+1}^N m^k s_k$, it must also divide $m^p s_p$.  This implies that $m$ divides $s_p$, which is a contradiction, for $0 < s_p < m$.  Therefore, there is no non-zero element in $\{s_k\}$.  This proves that $s_0 = s_1 = s_2 = \cdots = 0$.

The other direction of the proof is trivial. \end{proof}

\begin{corollary} \label{equal series} Let $\{s_k\}_{k=0}^\infty$ and $\{t_k\}_{k=0}^\infty$ be sequences of least nonnegative residues mod $m$.  Suppose $\sum_{k=0}^\infty m^k s_k$ and $\sum_{k=0}^\infty m^k t_k$ each converge to the same limit.  Then $s_k = t_k$ for every $k$.
\end{corollary}

\begin{proof} The hypotheses together imply that $\sum_{k=0}^\infty m^k (s_k - t_k) = 0$, so by Lemma \ref{0 series}, $s_k - t_k = 0$ for every $k$. \end{proof}

\begin{proposition} \label{not complete} The ring $D_m$ is not complete. \end{proposition}

\begin{proof} Suppose, in order to obtain a contradiction,  that every Cauchy sequence in $D_m$ converges.  Let $S$ be the set of all series in $D_m$ of the form $\sum_{k=0}^\infty m^k s_k$, where $\{s_k\}_{k=0}^\infty$ is a sequence of least nonnegative residues mod $m$.  Note that every series in $S$ is Cauchy, there are uncountably many distinct series in $S$, and Lemma \ref{equal series} implies that each one converges to a different limit.  Thus there are uncountably many points in $D_m$.  But $D_m$ is a subset of the rational numbers, so $D_m$ is countable.  This is a contradiction.  Hence $D_m$ is not complete. \end{proof}

\textbf{Example.}  It is a little more satisfying to see a concrete example of a Cauchy sequence in $D_m$ that fails to converge.  We now produce such a sequence.  Suppose $m = 3$.  Let $\sigma_1 = 2$ and for each positive integer $n$, let $\sigma_{n+1} = 7 - \sigma_n^2 + \sigma_n$.  The resulting sequence $\left\{\sigma_n\right\}_{n=1}^\infty$ is Cauchy, but it fails to converge.  Here are the first few terms:

\[2, 5, -13, -175, -30793, -948239635, -899158406333172853, \ldots \]

Under our metric in $D_m$, it is not immediately obvious why this does not converge.  In fact, we will show that the sequence of the squares of these terms, $\left\{\sigma_n^2\right\}_{n=1}^\infty$, which is even more divergent in the real number sense, does converge:

\begin{multline*}
4, 25, 169, 30625, 948208849, 899158405384933225, 808485839679611178962908785976159609, \ldots
\end{multline*}

The terms in this sequence diverge very quickly in the normal sense, but in $D_3$ they actually converge to 7.  One can see this by subtracting 7 from each term and factoring out 3's.  We can use the convergence of this sequence to show that $\left\{\sigma_n\right\}_{n=1}^\infty$ does not converge.  First, suppose it did converge; call the limit $\sigma$.  Then since multiplication is continuous in $D_3$, we have $\sigma^2 = 7$.  But since $\sigma$ is an element of $D_3$, it is expressible as a rational number.  Thus $\sigma$ is a rational number with $\sigma^2 = 7$.  This is a contradiction.  It follows that $\left\{\sigma_n\right\}_{n=1}^\infty$ does not converge in $D_3$.

We have not yet proved that $\left\{\sigma_n\right\}_{n=1}^\infty$ is Cauchy and that $\left\{\sigma_n^2\right\}_{n=1}^\infty$ actually converges.  We will prove a more general proposition instead.  We will use the following two lemmas frequently, but we omit proofs because they are trivial:

\begin{lemma} \label{decreasing-equiv} Suppose $k \leq n$.  If $x \equiv y \bmod{m^n}$, then $x \equiv y \bmod{m^k}$. \end{lemma}

\begin{lemma} \label{increasing-equiv} If $x \equiv y \bmod{m^n}$, and $z \equiv 0 \bmod{m^k}$, then $xz \equiv yz \bmod{m^{n+k}}$. \end{lemma}

We will also use the following lemma in order to define a sort of ``square root sequence" in $D_m$.

\begin{lemma} Suppose that $c$ is relatively prime to $m$ and that $\sigma_1^2 \equiv c \bmod{m}$.  Then $\sigma_1$ is relatively prime to $m$; hence $\sigma_1$ has an inverse modulo $m$. \end{lemma}

\begin{proof} Since $c$ is relatively prime to $m$, there exist integers $s$ and $t$ such that $sc + tm = 1$.  Since $\sigma_1^2 \equiv c \bmod{m}$, there exists some integer $k$ such that $\sigma_1^2 - c = km$.  Multiply both sides of this equation by $s$ to get $s\sigma_1^2 - sc = skm$, and substitute to get $s\sigma_1^2 - (1 - tm) = skm$.  Now rearrange to get $s\sigma_1^2 + (t-sk)m = 1$.  Thus $\sigma_1$ and $m$ are relatively prime. \end{proof}

Let us now define fairly general sequence in $D_m$ that is Cauchy but may not necessarily converge.

\begin{definition} Suppose that $m$ is odd and that $c$ is a number relatively prime to $m$ such that $x^2 \equiv c \bmod{m}$ has a solution $\sigma_1$.  For each positive integer $n$, let $\sigma_{n+1} = (c - \sigma_n^2)(2\sigma_1)^{-1} + \sigma_n$, where $(2\sigma_1)^{-1}$ means the inverse of $2\sigma_1$ modulo $m$.  \end{definition}

We will prove two propositions about this sequence.  First, we will show that it is a sort of ``square root sequence" in $D_m$.  Then, we will prove that the sequence is Cauchy.

\begin{proposition} \label{c^2} For every positive integer $n$, $\sigma_n^2 \equiv c \bmod{m^n}$.  Therefore, $\left\{\sigma_n^2\right\}_{n=1}^\infty$ converges to $c$.  \end{proposition}

\begin{proof} To build a proof by strong induction, first note that $\sigma_1^2 \equiv c \bmod{m^1}$ is given; now suppose $c - \sigma_k^2 \equiv 0 \bmod{m^k}$ for $1 \leq k \leq n$.  It follows that $\sigma_{k+1} \equiv \sigma_k \bmod{m^k}$ for $1 \leq k < n$.  Thus $\sigma_n \equiv \sigma_1 \bmod{m}$, and so $2\sigma_1$ and $2\sigma_n$ must have the same inverse mod $m$, i.e. $(2\sigma_1)^{-1} = (2\sigma_n)^{-1}$.  Therefore,
\begin{align*}
\sigma_{n+1}^2 &= ((c - \sigma_n^2)(2\sigma_1)^{-1} + \sigma_n)^2\\
&= ((c - \sigma_n^2)(2\sigma_n)^{-1} + \sigma_n)^2\\
&= ((c - \sigma_n^2)(2\sigma_1)^{-1})^2 + 2(c - \sigma_n^2)(2\sigma_n)^{-1}\sigma_n + \sigma_n^2\\
&\equiv (c - \sigma_n^2)(2\sigma_n)^{-1}(2\sigma_n) + \sigma_n^2 \bmod{m^{n+1}}.
\end{align*}
The squared term drops out because, as $c - \sigma_n^2 \equiv 0 \bmod{m^n}$, it follows that $(c - \sigma_n^2)^2 \equiv 0 \bmod{m^{2n}}$; hence $(c - \sigma_n^2)^2 \equiv 0 \bmod{m^{n+1}}$.  Now $(2\sigma_n)^{-1}(2\sigma_n) \equiv 1 \bmod{m}$, and since $c - \sigma_n^2 \equiv 0 \bmod{m^n}$, it follows that $(c - \sigma_n^2)(2\sigma_n)^{-1}(2\sigma_n) \equiv (c - \sigma_n^2) \bmod{m^{n+1}}$.  Therefore, we get $\sigma_{n+1}^2 \equiv (c - \sigma_n^2) + \sigma_n^2 = c \bmod{m^{n+1}}$.  By strong induction, $\sigma_n^2 \equiv c \bmod{m^n}$ for every positive integer $n$. \end{proof}

\begin{proposition} \label{cauchyness} The sequence $\{\sigma_n\}_{n=1}^\infty$ in $D_m$ is Cauchy. \end{proposition}

\begin{proof}  Notice that Proposition \ref{c^2} implies $\sigma_{n+1} - \sigma_n \equiv 0 \bmod{m^n}$ for every positive integer $n$.  Therefore, if $n \geq k$, then \begin{equation}\sigma_n - \sigma_k = (\sigma_n - \sigma_{n-1}) + (\sigma_{n-1} - \sigma_{n-2}) + \cdots + (\sigma_{k+1} - \sigma_k)\end{equation} is divisible by $m^k$; hence $\left|\sigma_n - \sigma_k\right|_m \leq m^{-k}$.  It follows that $\{\sigma_n\}_{n=1}^\infty$ is Cauchy. \end{proof}

Now if $c$ does not have a real square root, then Proposition \ref{c^2} and Proposition \ref{cauchyness} imply that $\left\{\sigma_n\right\}_{n=1}^\infty$ does not converge $D_m$, because if $\sigma_n \rightarrow \sigma$, then $\sigma^2 = c$, which is impossible.  (In the case $m = 3$, $c = 7$, we used $\sigma_1 = 2$ to get $(2\sigma_1)^{-1} = 1$.)  This gives infinitely many examples of Cauchy sequences in $D_m$ (for odd values of $m$) that fail to converge.

As a concluding remark to this section, we now have all we need in order for the sum $\sum_{k=1}^\infty \left(\frac{m}{m+1}\right)^k r_k$ in Theorem \ref{m-adic sum} to make sense.  However, in $D_m$ we do not have the comfort of being able to say that every sum of that form actually converges.  In the next section, we explore the Cauchy completion of $D_m$ so that we can operate in a world where $\sum_{k=1}^\infty \left(\frac{m}{m+1}\right)^k r_k$ converges even when the coefficients $\{r_k\}$ are chosen arbitrarily, independent of their meaning in the triangle $T_m$.  This gives us a deeper background against which we may view the unique behavior of actual column tracking sequences $\{r_k(x,r)\}$.

\section{The Ring $\mathbb{Q}_m$}

\begin{definition} Let $\mathbb{Q}_m$ be the Cauchy completion of $D_m$. \end{definition}

In this section, we want to show two ways to represent an element in $\mathbb{Q}_m$.  First, we establish that everything in $\mathbb{Q}_m$ can be represented in a familiar ``$m$-adic" way, i.e. in the form $\sum_{k=0}^\infty m^k s_k$.  Then, we show that everything in $\mathbb{Q}_m$ can be represented in an almost ``base $\frac{m}{m+1}$" way, i.e. in the form $\sum_{k=0}^\infty \left(\frac{m}{m+1}\right)^k s_k$.  This latter form is key because it appears in Theorem \ref{m-adic sum}.

Let us proceed with the first task.  First we show that every element of $D_m$ can be written in the form $\sum_{k=0}^\infty m^k s_k$, and then we use this fact to show that every element of $\mathbb{Q}_m$ can be written in this form.

\begin{lemma} \label{sequence of equivs} Let $a_0 = a$ be an integer, and let $b$ be an integer relatively prime to $m$.  Let $\{a_n\}_{n=0}^\infty$ be a sequence of integers and let $\{s_n\}_{n=0}^\infty$ be a sequence of least nonnegative residues mod $m$, satisfying the following:

\begin{enumerate} 
\item $s_nb \equiv a_n \bmod{m}$ for every $n$, and
\item $a_n = \frac{a_{n-1} - s_{n-1}b}{m}$ for every $n \geq 1$. \end{enumerate} Then $a \equiv \sum_{k=0}^n m^{k}s_kb \bmod{m^{n+1}}$ for every positive integer $n$.

\end{lemma}

\begin{proof} We claim that for any nonnegative integer $j$, if $n \geq j$, then \begin{equation}a_{n-j} \equiv \sum_{k=n-j}^n m^{k-n+j}s_kb \bmod{m^{j+1}}.\end{equation}  To prove this claim by induction, note that the case $j = 0$ is true by part (1) above: $a_n \equiv s_nb \bmod{m}$.  Now suppose the claim holds for an arbitrary nonnegative integer $j$.  Since $a_{n-j} = \frac{a_{n-j-1} - s_{n-j-1}b}{m}$, we have $a_{n-j-1} = ma_{n-j} + s_{n-j-1}b$.  By the inductive hypothesis, we derive \begin{equation}ma_{n-j} \equiv \sum_{k=n-j}^n m^{k-n+j+1}s_kb \bmod{m^{j+2}},\end{equation} and so \begin{equation}a_{n-j-1} \equiv \sum_{k=n-j-1}^n m^{k-n+j+1}s_kb \bmod{m^{j+2}}.\end{equation}  The claim is thus proved by induction.  Now take $n = j$ to get $a_0 \equiv \sum_{k=0}^n m^{k}s_kb \bmod{m^{n+1}}$ for every positive integer $n$. \end{proof}

\begin{proposition} \label{m^k D sum} Every element of $D_m$ can be represented in the form $\sum_{k=0}^\infty m^k s_k$ for a unique sequence $\left\{s_k\right\}_{k=0}^\infty$ of least nonnegative residues mod $m$. \end{proposition}

\begin{proof} Let $a/b$ be an element of $D_m$.  Then $b$ must be relatively prime to $m$.  Let $\{a_n\}_{n=0}^\infty$ and $\{s_n\}_{n=0}^\infty$ be defined as in Lemma \ref{sequence of equivs}.  Then $a \equiv \sum_{k=0}^n m^{k}s_kb$ mod $m^{n+1}$ for every positive integer $n$.  Thus $\left|a/b - \sum_{k=0}^n m^{k}s_k\right|_m < m^{-n}$ for every positive integer $n$.  It follows that $a/b = \sum_{k=0}^\infty m^k s_k$.  The uniqueness of $\left\{s_k\right\}_{k=0}^\infty$ is established by Corollary \ref{equal series}. \end{proof}

\begin{proposition} \label{m^k Q sum} Every element of $\mathbb{Q}_m$ can be represented in the form $\sum_{k=0}^\infty m^k s_k$ for a unique sequence $\left\{s_k\right\}_{k=0}^\infty$ of least nonnegative residues mod $m$. \end{proposition}

\begin{proof} Let $x$ be a point in $\mathbb{Q}_m$ and let $\{d_n\}_{n=0}^\infty$ be a sequence in $D_m$ converging to $x$.  Let $\{c_n\}_{n=0}^\infty$ be a subsequence of $\{d_n\}_{n=0}^\infty$ such that ${|c_n - x|_m \leq m^{-n-2}}$ for every positive integer $n$.  By the triangle inequality, \begin{equation*}|c_n - c_{n-1}|_m \leq |c_n - x|_m + |c_{n-1} - x|_m \leq m^{-n-2} + m^{-n-1} \leq m^{-n}.\end{equation*}  Using Proposition \ref{m^k D sum}, for each positive integer $n$ we represent $c_n$ as a series $\sum_{k=0}^\infty m^k s_k(n)$ for a unique sequence $\{s_k(n)\}_{k=0}^\infty$ of least nonnegative residues mod $m$.  So for each positive integer $n$, \begin{equation*}\left|\sum_{k=0}^\infty m^k (s_k(n)-s_k(n-1))\right|_m \leq m^{-n},\end{equation*} which implies that $s_k(n) - s_k(n-1) = 0 $ for $1 \leq k \leq n-1$.

We will now show that the limit of the series $\sum_{k=0}^\infty m^k s_k(k)$, which is Cauchy and therefore converges in $\mathbb{Q}_m$, is equal to $x$.  Observe that for any $n$, if $k \leq n$, then the fact that $s_k(n-1) = s_k(n)$ whenever $1 \leq k \leq n-1$ can be used repeatedly to show that \begin{equation*}s_k(k) = s_k(k+1) = s_k(k+2) = \cdots = s_k(n-1).\end{equation*}  It follows that \begin{equation*}\left|\sum_{k=0}^\infty m^k (s_k(n)-s_k(k))\right|_m \leq m^{-n}, \text{ or } \left|c_n - \sum_{k=0}^\infty m^k s_k(k)\right|_m \leq m^{-n}\end{equation*} for each $n$.  We have
\begin{equation*}
\left|\sum_{k=0}^\infty m^k s_k(k) - x\right|_m \leq \left|c_n - \sum_{k=0}^\infty m^k s_k(k)\right|_m + \left|c_n - x\right|_m \leq m^{-n} + m^{-n-2} \leq m^{-n+1}
\end{equation*} 
by the triangle inequality.  Since this inequality holds for any $n$, it follows that 
$\sum_{k=0}^\infty m^k s_k = x$.

The uniqueness of $\left\{s_k\right\}_{k=0}^\infty$ is once again established by Corollary \ref{equal series}.
\end{proof}

Now we move onto the next task: show that every element in $\mathbb{Q}_m$ can be represented as a sum of this form: \[\sum_{k=0}^\infty \left(\frac{m}{m+1}\right)^k s_k.\]
This kind of sum appears in Theorem \ref{m-adic sum}.  To show that every element of $\mathbb{Q}_m$ can be written in this form, we use an argument analogous to our previous argument showing every element of $\mathbb{Q}_m$ can be written in the ordinary ``$m$-adic" way.

Because we will also want to show uniqueness of such representations, we need the following lemma, analogous to Lemma \ref{0 series}, in order to show the uniqueness of these representations.

\begin{lemma} \label{0 series plus} Let $\{s_k\}_{k=0}^\infty$ be a sequence of least nonnegative residues mod $m$.  Then the series $\sum_{k=0}^\infty \left(\frac{m}{m+1}\right)^k s_k$ converges to 0 in $D_m$ if and only if $s_0 = s_1 = s_2 = \cdots = 0$. \end{lemma}

\begin{proof} Suppose $\sum_{k=0}^\infty \left(\frac{m}{m+1}\right)^k s_k = 0$. Suppose, in order to obtain a contradiction,  that there are non-zero elements in $\{s_k\}$.  Let $p$ be the smallest integer such that $s_p \neq 0$.  Since $\sum_{k=0}^\infty \left(\frac{m}{m+1}\right)^k s_k = 0$, there exists some $N_1$ such that whenever $n \geq N_1$, $\left| \sum_{k=0}^n \left(\frac{m}{m+1}\right)^k s_k \right|_m < m^{-p}$.  Let $N = \min\{p, N_1\}$.  Now

\begin{align*}
m^{-p} &> \left| \sum_{k=0}^N \left(\frac{m}{m+1}\right)^k s_k \right|_m\\
&= \left|\left(\frac{m}{m+1}\right)^p s_p + \left(\frac{m}{m+1}\right)^{p+1}s_{p+1} + \cdots + \left(\frac{m}{m+1}\right)^N s_N \right|_m\\
&= \left|\frac{m^p(m+1)^{N-p}s_p + m^{p+1}(m+1)^{N-(p+1)}s_{p+1} + \cdots + m^Ns_N}{(m+1)^N} \right|_m,
\end{align*}
which means that $m^{p+1}$ must divide \[m^p(m+1)^{N-p}s_p + m^{p+1}(m+1)^{N-(p+1)}s_{p+1} + \cdots + m^Ns_N.\]  Since $m^{p+1}$ divides \[m^{p+1}(m+1)^{N-(p+1)}s_{p+1} + \cdots + m^Ns_N,\] it must also divide $m^p(m+1)^{N-p}s_p$.  This implies that $m$ divides $s_p$, which is a contradiction, for $0 < s_p < m$.  Therefore, there is no non-zero element in $\{s_k\}$.  This proves that $s_0 = s_1 = s_2 = \cdots = 0$.

The other direction of the proof is trivial. \end{proof}

\begin{corollary} \label{equal series plus} Let $\{s_k\}_{k=0}^\infty$ and $\{t_k\}_{k=0}^\infty$ be sequences of least nonnegative residues mod $m$.  Suppose $\sum_{k=0}^\infty \left(\frac{m}{m+1}\right)^k s_k$ and $\sum_{k=0}^\infty \left(\frac{m}{m+1}\right)^k t_k$ each converge to the same limit.  Then $s_k = t_k$ for every $k$.
\end{corollary}

\begin{proof} The hypotheses together imply that $\sum_{k=0}^\infty \left(\frac{m}{m+1}\right)^k (s_k - t_k) = 0$, so by Lemma \ref{0 series plus}, $s_k - t_k = 0$ for every $k$. \end{proof}

\begin{lemma} \label{sequence of equivs plus} Let $a_0 = a$ be an integer, and let $b$ be an integer relatively prime to $m$.  Let $\{a_n\}_{n=0}^\infty$ be a sequence of integers and let $\{s_n\}_{n=0}^\infty$ be a sequence of least nonnegative residues mod $m$, satisfying the following:

\begin{enumerate} 
\item $s_nb \equiv a_n$ mod $m$ for every $n$, and
\item $a_n = (m+1)\frac{a_{n-1} - s_{n-1}b}{m}$ for every $n \geq 1$. \end{enumerate} Then $(m+1)^n a \equiv \sum_{k=0}^n m^{k}(m+1)^{n-k}s_kb$ mod $m^{n+1}$ for every positive integer $n$.

\end{lemma}

\begin{proof} We claim that for any nonnegative integer $j$, if $n \geq j$, then \begin{equation}(m+1)^j a_{n-j} \equiv \sum_{k=n-j}^n m^{k-n+j}(m+1)^{n-k}s_kb \text{ mod } m^{j+1}.\end{equation}  To prove this claim by induction, note that the case $j = 0$ is true by part (1) above: $a_n \equiv s_nb$ mod $m$.  Now suppose the claim holds for an arbitrary nonnegative integer $j$.  Since $a_{n-j} = (m+1)\frac{a_{n-j-1} - s_{n-j-1}b}{m}$, we have $(m+1)^{j+1}a_{n-j-1} = m(m+1)^ja_{n-j} + (m+1)^js_{n-j-1}b$.  By the inductive hypothesis, we derive \begin{equation}m(m+1)^ja_{n-j} \equiv \sum_{k=n-j}^n m^{k-n+j+1}(m+1)^{n-k}s_kb \text{ mod } m^{j+2},\end{equation} and so \begin{equation}(m+1)^{j+1} a_{n-j-1} \equiv \sum_{k=n-j-1}^n m^{k-n+j+1}(m+1)^{n-k}s_kb \text{ mod } m^{j+2}.\end{equation}  The claim is thus proved by induction.  Now take $n = j$ to get $(m+1)^n a_0 \equiv \sum_{k=0}^n m^{k}(m+1)^{n-k}s_kb$ mod $m^{n+1}$ for every positive integer $n$. \end{proof}

\begin{lemma} \label{our D sum} Every element of $D_m$ can be represented in the form \[\sum_{k=0}^\infty \left(\frac{m}{m+1}\right)^k s_k\] for a unique sequence $\left\{s_k\right\}_{k=0}^\infty$ of least nonnegative residues mod $m$. \end{lemma}

\begin{proof} Let $a/b$ be an element of $D_m$.  Then $b$ must be relatively prime to $m$.  Let $\{a_n\}_{n=0}^\infty$ and $\{s_n\}_{n=0}^\infty$ be defined as in Lemma \ref{sequence of equivs plus}.  Then \begin{equation}(m+1)^n a \equiv \sum_{k=0}^n m^{k}(m+1)^{n-k}s_kb \text{ mod } m^{n+1}\end{equation} for every positive integer $n$.  Thus $\left|a/b - \sum_{k=0}^n \left(\frac{m}{m+1}\right)^{k}s_k\right|_m < m^{-n}$ for every positive integer $n$.  It follows that $a/b = \sum_{k=0}^\infty \left(\frac{m}{m+1}\right)^{k}s_k$.  The uniqueness of $\left\{s_k\right\}_{k=0}^\infty$ is established by Corollary \ref{equal series}. \end{proof}

\begin{proposition} \label{our Q sum} Every element of $\mathbb{Q}_m$ can be represented in the form $\sum_{k=0}^\infty \left(\frac{m}{m+1}\right)^k s_k$ for a unique sequence $\left\{s_k\right\}_{k=0}^\infty$ of least nonnegative residues mod $m$. \end{proposition}

\begin{proof} This proof follows the proof of Proposition \ref{m^k Q sum} exactly to produce a series that will converge to a given point in $\mathbb{Q}_m$.  To show uniqueness, we use Lemma \ref{0 series plus}. \end{proof}

\section{Rotation remainder expansions}  We proved that every element of $\mathbb{Q}_m$ can be represented in the form $\sum_{k=0}^\infty \left(\frac{m}{m+1}\right)^k s_k$ for a unique sequence $\{s_k\}_{k=0}^\infty$ of least nonnegative residues mod $m$.  This kind of sum appeared in Theorem \ref{m-adic sum}.  So, what do these sequences $\{s_k\}_{k=0}^\infty$ look like, and how can we see that they represent the number we say they represent?  Furthermore, we know that the sum form $\sum_{k=0}^\infty \left(\frac{m}{m+1}\right)^k s_k$ appears in Theorem \ref{m-adic sum}, but how exactly do these representations apply to the problems posed in this paper?  In this section we will use concrete examples to answer these questions.  We will restrict ourselves to the case $m = 3$ because this is a case with sufficient complexity to reveal all the basic patterns and difficulties of the general problem.  But first, let us establish a consistent notation to match elements in $\mathbb{Q}_m$ with their expansions as sums.

\begin{definition} We define a function $R$ mapping $\mathbb{Q}_m$ into sequences of least nonnegative residues mod $m$ as follows.  Let $q \in \mathbb{Q}_m$ be given.  Suppose $\{s_k\}_{k=0}^\infty$ is a sequence of least nonnegative residues mod $m$ such that $q = \sum_{k=0}^\infty \left(\frac{m}{m+1}\right)^k s_k$.  Then we say $R(q) = \{s_k\}_{k=0}^\infty$.  This is called the \emph{rotation remainder expansion} of $q$. \end{definition}

\begin{remark} First, we note that $R$ is a bijection, with $R^{-1}$ mapping any sequence $\{s_k\}_{k=0}^\infty$ of least nonnegative residues mod $m$ to the point $q = \sum_{k=0}^\infty \left(\frac{m}{m+1}\right)^k s_k$ in $\mathbb{Q}_m$.  Second, we may refer to the terms $s_k$ in the sequence $R(q) = \{s_k\}_{k=0}^\infty$ as ``digits" of the rotational remainder expansion of $q$.  Since we start with $s_0$, we will likely refer to $s_{k-1}$ as the $k$th digit in the expansion. \end{remark}

\subsection*{Relating expansions to $T_m$}

A concrete example of a rotation remainder expansion will help relate this concept to the original problem posed in the introduction.  Consider \begin{equation*}R(16) = \{1, 2, 0, 2, 1, 1, 2, 1, 2, 2, 2, 0, \ldots\},\end{equation*} which the reader will recognize as $\{r_n(1,0)\}_{n=1}^\infty$. (See Equation \eqref{1 sequence}) Thus the sequence $R(16)$ identifies the column positions in which the 1 appears if we restrict our view to the first three columns.  This example is just an illustration of the fact that we can rewrite Theorem \ref{m-adic sum} in the following way:

\begin{theorem} \label{m-adic sum restated} In $\mathbb{Q}_m$, for any starting row $x$ and column $r$, $R((m+1)y_0(x,r)) = \{r_{k+1}(x,r)\}_{k=0}^\infty$. \end{theorem}

\subsection*{Producing expansions}

Now we want to understand how to compute rotation remainder expansions.  Consider again $R(16)$, as depicted above.  By definition, this sequence represents the integer 16 in the sense that \begin{equation}16 = 1 + 2\left(\frac{3}{4}\right) + 0\left(\frac{9}{16}\right) + 2\left(\frac{27}{64}\right) + \cdots,\end{equation} where these coefficients in this power series come from $R(16)$.  Let us consider what this means in $\mathbb{Q}_3$.

\begin{enumerate}
\item $16 - 1 = 15$, which is divisible by 3;

\item $16 - (1 + 2\left(\frac{3}{4}\right)) = \frac{54}{4}$, whose numerator is divisible by $3^2$;

\item $16 - (1 + 2\left(\frac{3}{4}\right) + 0\left(\frac{9}{16}\right)) = \frac{216}{16}$, whose numerator is divisible by $3^3$;

\item $16 - (1 + 2\left(\frac{3}{4}\right) + 0\left(\frac{9}{16}\right) + 2\left(\frac{27}{64}\right)) = \frac{810}{64}$, whose numerator is divisible by $3^4$;\\ and so on.

\end{enumerate}

Given the norm used in $\mathbb{Q}_3$, this progression shows that the series converges to 16 because each successive partial sum differs from 16 by a greater power of 3.

This is a rather exotic expansion to represent the integer 16.  However, we have proved that this is the only possible sequence that could represent 16 using least nonnegative residues mod 3.  How would one compute this expansion without looking at the triangle $T_3$?  We would choose successive digits one at a time to get the pattern shown above:

\begin{enumerate}

\item We desire $3 \mid 16 - s_0$.  Should we choose $s_0 = 0, 1,$ or 2?  Clearly $s_0 = 1$.

\item Now we desire $3^2 \mid 4(16 - s_0) - 3s_1$.  Should we choose $s_1 = 0, 1,$ or 2?  We must choose $s_1 = 2$ to get $4(16 - s_0) - 3s_1 = 54$, the numerator of the second partial sum in our series.

\item Now we desire $3^3 \mid 4(4(16 - s_0) - 3s_1) - 3^2s_2$.  Should we choose $s_2 = 0, 1,$ or 2?  We must choose $s_2 = 0$ to get $4(4(16 - s_0) - 3s_1) - 3^2s_2 = 216$, the numerator of the second partial sum in our series.

Continue in this fashion.  At each step, the next digit is uniquely determined by a linear equation modulo $3^k$, where $k$ increases by 1 every step.
\end{enumerate}

We can use this technique to get $R(q)$ for any rational number $q$.  The process of deriving $R(q)$ for an irrational number $q \in \mathbb{Q}_m$ involves applying this process to each term in a sequence of rationals that converges to $q$ in $\mathbb{Q}_m$.  See the example from ``The Ring $D_m$" on sequences converging to square roots.   For now, let us focus on integers $q \in \mathbb{Q}_m$.  There are some very simple expansions to look at before looking at all the exotic ones.  Here are a few examples:

\begin{align*}
R(0) &= \{0, 0, 0, 0, 0, 0, 0, 0, 0, 0, 0, 0, 0, 0, 0, 0, 0, 0, 0, 0, 0,\ldots\}\\
R(1) &= \{1, 0, 0, 0, 0, 0, 0, 0, 0, 0, 0, 0, 0, 0, 0, 0, 0, 0, 0, 0, 0,\ldots\}\\
R(2) &= \{2, 0, 0, 0, 0, 0, 0, 0, 0, 0, 0, 0, 0, 0, 0, 0, 0, 0, 0, 0, 0,\ldots\}\\
R(3) &= \{0, 1, 1, 1, 1, 1, 1, 1, 1, 1, 1, 1, 1, 1, 1, 1, 1, 1, 1, 1, 1,\ldots\}\\
R(4) &= \{1, 1, 1, 1, 1, 1, 1, 1, 1, 1, 1, 1, 1, 1, 1, 1, 1, 1, 1, 1, 1,\ldots\}\\
R(5) &= \{2, 1, 1, 1, 1, 1, 1, 1, 1, 1, 1, 1, 1, 1, 1, 1, 1, 1, 1, 1, 1,\ldots\}\\
R(6) &= \{0, 2, 2, 2, 2, 2, 2, 2, 2, 2, 2, 2, 2, 2, 2, 2, 2, 2, 2, 2, 2,\ldots\}\\
R(7) &= \{1, 2, 2, 2, 2, 2, 2, 2, 2, 2, 2, 2, 2, 2, 2, 2, 2, 2, 2, 2, 2,\ldots\}\\
R(8) &= \{2, 2, 2, 2, 2, 2, 2, 2, 2, 2, 2, 2, 2, 2, 2, 2, 2, 2, 2, 2, 2,\ldots\}\\
R(9) &= \{0, 0, 1, 2, 0, 2, 1, 1, 2, 1, 2, 2, 2, 0, 1, 0, 0, 1, 1, 0, 0,\ldots\}\\
& \ldots
\end{align*}

There are three casual observations worth making here.  First observation: all these sequences are eventually constant except for $R(9)$, which is aperiodic (unfortunately the ellipsis fails to capture this).  A constant sequence $R(q)$ means $q$ can be written as a geometric series in $\mathbb{Q}_3$.  For instance, we have \begin{equation} 4 = 1 + \frac{3}{4} + \frac{9}{16} + \frac{27}{64} + \cdots. \end{equation} This equation should not be surprising, since it holds true in $\mathbb{R}$ as well as in $\mathbb{Q}_3$.  The reason why it holds in $\mathbb{Q}_3$ is a little different than in $\mathbb{R}$.  The partial sums of this series can be written in the condensed form $\frac{1 - \left(\frac{3}{4}\right)^k}{1 - \frac{3}{4}} = 4\left(1 - \left(\frac{3}{4}\right)^k \right)$ for $k = 1, 2, 3, \ldots$.  As $k$ gets large, $\left(\frac{3}{4}\right)^k$ converges to 0 in $\mathbb{Q}_3$ because the numerator in this expression contains an increasingly large power of 3.  Thus the partial sums of $1 + \frac{3}{4} + \frac{9}{16} + \frac{27}{64} + \cdots$ converge to $4(1 - 0) = 4$.  In this way we see that whenever $R(q)$ is eventually periodic or constant, it follows that $q$ is rational, and in the real number sense, $0 \leq q \leq 8$.  The shocking result is that whenever a real number $q \in \mathbb{Q}_3$ is greater than 8 as compared in the real numbers, then $R(q)$ must be aperiodic.  Though lacking some details, this is one possible way to prove Theorem \ref{aperiodic}.

Second observation: the tail of $R(9)$ starting with the third term is exactly the same as $R(16)$.  It is not hard to see why this is so: if we multiply $q = \sum_{k=0}^\infty \left(\frac{m}{m+1}\right)^k s_k$ by $3/4$, we get a sum of the same form with indices shifted up by one and a zero taking the 0 index.  Indeed, $R(12)$ looks like $\{0, 1, 2, 0, 2, 1, 1, 2, 1, 2, 2, 2, 0, 1, 0,\ldots\}$, the result of multiplying the series representing 16 by $3/4$ once, and to get $R(9)$ we multiply by $3/4$ again.  In general, $R(3q/4)$ is just the sequence $R(q)$ with a 0 added onto the beginning.  This illustrates a similarity between rotation remainder expansions and decimal expansions: multiplying by the base of the number system (in decimals, the base is 10; in this case, the ``base" is $3/4$) adds a zero to the expansion and shifts the digits over by one place.

Third observation: as we continually add 1, we see the first digit change in a predictable way.  We start with 0, move up to 1, then to 2, then jump back down to 0 again, and repeat.  When we jump back down to 0 in the first digit, we expect that then the second digit will be affected by a sort of ``carry," just as in decimals, $.01 + .09 = .10$.  We do find this carry in rotation remainder expansions, but we also find that \emph{all} the digits beyond the first digit are affected by this carry.  Thus $R(0), R(3),$ and $R(6)$, though all equivalent mod 3, have \emph{nothing} in common past the first digit.  This illustrates a difference between rotational remainder expansions and decimal expansions: in decimal expansions, adding 1 continually generally leaves most digits intact, whereas adding 1 continually in rotation remainder expansions will periodically change all of the digits beyond the first one.

\subsection*{Arithmetic on expansions}

Since these sequences represent integers, one has to wonder how to manipulate these sequences like we manipulate integers -- adding, subtracting, multiplying, and dividing.  We will now explore addition and multiplication.

As a first hint at how to add these sequences, we return to the third observation made above: as we add 1 repeatedly, there is at some points a ``carry" that affects not only the second digit, but all the digits beyond the first.  For instance, upon adding $R(1)$ and $R(2)$ to get $R(3)$, all of the digits beyond the first become 1's.  It turns out that this cumulative carry happens in all cases.  Here are a couple of examples to illustrate.

\[\begin{array}{ccccccccccccccccccc}
  \ & \ & \multicolumn{3}{c}{\text{\tiny{Carry:}}} \text{\tiny{0}} & \text{\tiny{1}} & \text{\tiny{2}} & \text{\tiny{3}} & \text{\tiny{5}} & \text{\tiny{7}} & \text{\tiny{10}} & \text{\tiny{14}} & \text{\tiny{19}} & \text{\tiny{26}} & \text{\tiny{35}} & \text{\tiny{48}} & \text{\tiny{65}}  & \ldots & \ \\
  \ & R(1) & = & \{ & 1 & 0, & 0, & 0, & 0, & 0, & 0, & 0, & 0, & 0, & 0, & 0, & 0, & \ldots & \} \\
  + & R(8) & = & \{ & 2, & 2, & 2, & 2, & 2, & 2, & 2, & 2, & 2, & 2, & 2, & 2, & 2, & \ldots & \} \\
  \hline
  = & R(9) & = & \{ & 0, & 0, & 1, & 2, & 0, & 2, & 1, & 1, & 2, & 1, & 2, & 2, & 2, & \ldots & \} \\
\end{array}\]

In this first example, 1 and 2 added in the first position to get 3, which we divided by 3 to get a carry of 1 and a remainder of 0.  Then in the second position, the carry of 1 was added to the digits 0 and 2 to get 3, which we divided by 3 to get a carry of 1 and a remainder of 0.  This carry was added on to the carry already accumulated, so in the third position, we had a total carry of 2.  This was added to 0 and 2 to get 4, which left a carry of 1 and a remainder of 1.  The process of adding continues in this way, almost exactly like decimal addition except that the carry is cumulative.  Here is another example.
\[\begin{array}{ccccccccccccccccccc}
  \ & \ & \multicolumn{3}{c}{\text{\tiny{Carry:}}} \text{\tiny{0}} & \text{\tiny{0}} & \text{\tiny{0}} & \text{\tiny{0}} & \text{\tiny{0}} & \text{\tiny{1}} & \text{\tiny{1}} & \text{\tiny{2}} & \text{\tiny{3}} & \text{\tiny{4}} & \text{\tiny{5}} & \text{\tiny{7}} & \text{\tiny{\text{\tiny{10}}}}  & \ldots & \ \\
  \ & R(34) & = & \{ & 1, & 1, & 1, & 0, & 2, & 0, & 1, & 2, & 0, & 0, & 0, & 2, & 0, & \ldots & \} \\
  + & R(25) & = & \{ & 1, & 1, & 0, & 2, & 2, & 1, & 2, & 1, & 2, & 1, & 2, & 2, & 2, & \ldots & \} \\
  \hline
  = & R(59) & = & \{ & 2, & 2, & 1, & 2, & 1, & 2, & 1, & 2, & 2, & 2, & 1, & 2, & 0, & \ldots & \} \\
\end{array}\]

The following proposition gives the formal procedure for adding two of these sequences.  We no longer restrict ourselves to $m = 3$, but rather we look at $\mathbb{Q}_m$ in general.

\begin{proposition} Let $a$ and $b$ be elements in $\mathbb{Q}_m$, and let $c = a + b$.  Suppose $R(a) = \{s_n(a)\}_{n=0}^\infty$, $R(b) = \{s_n(b)\}_{n=0}^\infty$, and $R(c) = \{s_n(c)\}_{n=0}^\infty$.  Define a ``carry" sequence as follows: let $\kappa_0 = 0$ and for $n \geq 0$ let \begin{equation} \kappa_{n+1} = \kappa_n + \left\lfloor \frac{s_n(a) + s_n(b) + \kappa_n}{m} \right\rfloor. \end{equation} Then $s_n(c) \equiv s_n(a) + s_n(b) + \kappa_n \bmod{m}$ for every $n$. \end{proposition}

\begin{proof} We will not prove this in general, but rather limit ourselves to the case where $a$ and $b$ are integers.  The proof generalizes to all cases when sufficient details are added.

We offer a proof by induction.  First, let $\{a_n\}$ be defined as in Lemma \ref{sequence of equivs plus}; that is, let $a_0 = a$ and let $a_n = (m+1)\frac{a_{n-1} - s_{n-1}(a)}{m}$ for every $n \geq 1$.  Define $\{b_n\}$ and $\{c_n\}$ analogously.  By that same Lemma, we know that $a_n \equiv s_n(a)$ mod $m$ for every $n$, and likewise for $b_n$ and $c_n$.  We will prove that $c_n = a_n + b_n + (m+1)\kappa_n$ by induction.

For $n = 0$, we have the obvious equation $c = a + b$.  Now assume the equation holds for an arbitrary $n \geq 0$.  Reducing the equation mod $m$, we see that $s_n(c) \equiv s_n(a) + s_n(b) + \kappa_n \bmod{m}$, and since $s_n(c)$ is a least residue mod $m$, it follows that \begin{equation} \left\lfloor \frac{s_n(a) + s_n(b) + \kappa_n}{m} \right\rfloor = \frac{s_n(a) + s_n(b) + \kappa_n - s_n(c)}{m}. \end{equation}  Now observe that
\begin{align*}
a_{n+1} + b_{n+1} + (m+1)\kappa_{n+1} &= a_{n+1} + b_{n+1} + (m+1)\kappa_n\\
& \hspace{1 cm} + (m+1)\frac{s_n(a) + s_n(b) + \kappa_n - s_n(c)}{m}\\
&= (m+1)\frac{a_{n} - s_{n}(a)}{m} + (m+1)\frac{b_{n} - s_{n}(b)}{m}\\
& \hspace{1 cm} + (m+1)\kappa_n + (m+1)\frac{s_n(a) + s_n(b) + \kappa_n - s_n(c)}{m}\\
&= (m+1)\frac{a_n + b_n + (m+1)\kappa_n - s_n(c)}{m}\\
&= (m+1)\frac{c_n - s_n(c)}{m}\\
&= c_{n+1},
\end{align*}
as desired.  Therefore, the equation holds for all $n$.  Now reduce the equation mod $m$ to get $s_n(c) \equiv s_n(a) + s_n(b) + \kappa_n \bmod{m}$ for all $n$, as desired. \end{proof}

This algorithm can be used on as many addends as desired.  This makes a multiplication algorithm easy to deduce.

Observe that \begin{align*} \left(\sum_{k=0}^\infty \left(\frac{m}{m+1}\right)^k s_k(a) \right) &\left(\sum_{k=0}^\infty \left(\frac{m}{m+1}\right)^k s_k(b) \right) \\= s_0(b)\left(\sum_{k=0}^\infty \left(\frac{m}{m+1}\right)^k s_k(a)\right)& + s_1(b)\left(\sum_{k=1}^\infty \left(\frac{m}{m+1}\right)^k  s_k(a)\right) + \cdots. \end{align*} This suggests that if we simply learn our ``times tables" for a given sequence, learning to multiply it by each of the integers $0,1,\ldots,m-1$, then we can multiply it by any other sequence we want.  With $m = 3$, we only need to learn how to double a sequence, and then we can multiply it by any other sequence.  For example, to multiply $R(9)$ by $R(11)$ to get $R(99)$:

\[\begin{array}{ccccccccccccccccccc}
    \ & R(9) & = & \{ & 0, & 0, & 1, & 2, & 0, & 2, & 1, & 1, & 2, & 1, & 2, & 2, & 2, & \ldots & \} \\
  \times & R(11) & = & \{ & 2, & 0, & 1, & 2, & 0, & 2, & 1, & 1, & 2, & 1, & 2, & 2, & 2, & \ldots & \} \\
  \hline
  = & R(99) & = & \{ & 0, & 0, & 2, & 1, & 2, & 0, & 1, & 1, & 1, & 0, & 2, & 0, & 1, & \ldots & \} \\
\end{array}\]
we use the following multiplication algorithm.  First, we need to double $R(9)$ to get $R(18)$ using just the addition algorithm:

\[\begin{array}{ccccccccccccccccccc}
  \ & \ & \multicolumn{3}{c}{\text{\tiny{Carry:}}} \text{\tiny{0}} & \text{\tiny{0}} & \text{\tiny{0}} & \text{\tiny{0}} & \text{\tiny{1}} & \text{\tiny{1}} & \text{\tiny{2}} & \text{\tiny{3}} & \text{\tiny{4}} & \text{\tiny{6}} & \text{\tiny{8}} & \text{\tiny{12}} & \text{\tiny{\text{\tiny{17}}}}  & \ldots & \ \\
  \ & R(9) & = & \{ & 0, & 0, & 1, & 2, & 0, & 2, & 1, & 1, & 2, & 1, & 2, & 2, & 2, & \ldots & \} \\
  + & R(9) & = & \{ & 0, & 0, & 1, & 2, & 0, & 2, & 1, & 1, & 2, & 1, & 2, & 2, & 2, & \ldots & \} \\
  \hline
  = & R(18) & = & \{ & 0, & 0, & 2, & 1, & 1, & 2, & 1, & 2, & 2, & 2, & 0, & 1, & 0, & \ldots & \} \\
\end{array}\]

Now we can multiply $R(9)$ by $R(11)$ just like decimals: multiply $R(9)$ by each digit in $R(11)$, including a right shift just as you would in decimal multiplication (note the second observation under ``Producing expansions").

\[\begin{array}{ccccccccccccccccc}
  \multicolumn{3}{c}{\text{\tiny{Carry:}}} \text{\tiny{0}} & \text{\tiny{0}} & \text{\tiny{0}} & \text{\tiny{0}} & \text{\tiny{0}} & \text{\tiny{0}} & \text{\tiny{2}} & \text{\tiny{3}} & \text{\tiny{6}} & \text{\tiny{10}} & \text{\tiny{16}} & \text{\tiny{24}} & \text{\tiny{\text{\tiny{35}}}}  & \ldots & \ \\
  \times & \{ & 0, & 0, & 1, & 2, & 0, & 2, & 1, & 1, & 2, & 1, & 2, & 2, & 2, & \ldots & \} \\
  \overbrace{2} & \ & 0, & 0, & 2, & 1, & 1, & 2, & 1, & 2, & 2, & 2, & 0, & 1, & 0, & \ldots & \ \\
  0 & \ & 0, & 0, & 0, & 0, & 0, & 0, & 0, & 0, & 0, & 0, & 0, & 0, & 0, & \ldots & \ \\
  1 & \ & 0, & 0, & 0, & 0, & 1, & 2, & 0, & 2, & 1, & 1, & 2, & 1, & 2, & \ldots & \ \\
  2 & \ & 0, & 0, & 0, & 0, & 0, & 2, & 1, & 1, & 2, & 1, & 2, & 2, & 2, & \ldots & \ \\
  0 & \ & 0, & 0, & 0, & 0, & 0, & 0, & 0, & 0, & 0, & 0, & 0, & 0, & 0, & \ldots & \ \\
  2 & \ & 0, & 0, & 0, & 0, & 0, & 0, & 0, & 2, & 1, & 1, & 2, & 1, & 2, & \ldots & \ \\
  1 & \ & 0, & 0, & 0, & 0, & 0, & 0, & 0, & 0, & 1, & 2, & 0, & 2, & 1, & \ldots & \ \\
  1 & \ & 0, & 0, & 0, & 0, & 0, & 0, & 0, & 0, & 0, & 1, & 2, & 0, & 2, & \ldots & \ \\
  2 & \ & 0, & 0, & 0, & 0, & 0, & 0, & 0, & 0, & 0, & 0, & 2, & 1, & 1, & \ldots & \ \\
  1 & \ & 0, & 0, & 0, & 0, & 0, & 0, & 0, & 0, & 0, & 0, & 0, & 1, & 2, & \ldots & \ \\
  2 & \ & 0, & 0, & 0, & 0, & 0, & 0, & 0, & 0, & 0, & 0, & 0, & 0, & 2, & \ldots & \ \\
  \underbrace{\vdots} & \ & \ & \ & \ & \ & \ & \ & \ & \ & \ & \ & \ & \ & \ & \ & \ \\
  \hline
  = & \{ & 0, & 0, & 2, & 1, & 2, & 0, & 1, & 1, & 1, & 0, & 2, & 0, & 1, & \ldots & \}
\end{array}\]

And we get the desired result.  Note that it is very easy to multiply by $R(q)$ for integers $q$ in the range $0 \leq q \leq 8$ (in the real number sense).  For instance, multiplying $R(4)$ by itself is easy, but somewhat enlightening:

\[\begin{array}{ccccccccccccccccc}
  \multicolumn{3}{c}{\text{\tiny{Carry:}}} \text{\tiny{0}} & \text{\tiny{0}} & \text{\tiny{0}} & \text{\tiny{1}} & \text{\tiny{2}} & \text{\tiny{4}} & \text{\tiny{7}} & \text{\tiny{11}} & \text{\tiny{17}} & \text{\tiny{25}} & \text{\tiny{36}} & \text{\tiny{51}} & \text{\tiny{\text{\tiny{72}}}}  & \ldots & \ \\
  \times & \{ & 1, & 1, & 1, & 1, & 1, & 1, & 1, & 1, & 1, & 1, & 1, & 1, & 1, & \ldots & \} \\
  \overbrace{1} & \ & 1, & 1, & 1, & 1, & 1, & 1, & 1, & 1, & 1, & 1, & 1, & 1, & 1, & \ldots & \ \\
  1 & \ & 0, & 1, & 1, & 1, & 1, & 1, & 1, & 1, & 1, & 1, & 1, & 1, & 1, & \ldots & \ \\
  1 & \ & 0, & 0, & 1, & 1, & 1, & 1, & 1, & 1, & 1, & 1, & 1, & 1, & 1, & \ldots & \ \\
  1 & \ & 0, & 0, & 0, & 1, & 1, & 1, & 1, & 1, & 1, & 1, & 1, & 1, & 1, & \ldots & \ \\
  1 & \ & 0, & 0, & 0, & 0, & 1, & 1, & 1, & 1, & 1, & 1, & 1, & 1, & 1, & \ldots & \ \\
  1 & \ & 0, & 0, & 0, & 0, & 0, & 1, & 1, & 1, & 1, & 1, & 1, & 1, & 1, & \ldots & \ \\
  1 & \ & 0, & 0, & 0, & 0, & 0, & 0, & 1, & 1, & 1, & 1, & 1, & 1, & 1, & \ldots & \ \\
  1 & \ & 0, & 0, & 0, & 0, & 0, & 0, & 0, & 1, & 1, & 1, & 1, & 1, & 1, & \ldots & \ \\
  1 & \ & 0, & 0, & 0, & 0, & 0, & 0, & 0, & 0, & 1, & 1, & 1, & 1, & 1, & \ldots & \ \\
  1 & \ & 0, & 0, & 0, & 0, & 0, & 0, & 0, & 0, & 0, & 1, & 1, & 1, & 1, & \ldots & \ \\
  1 & \ & 0, & 0, & 0, & 0, & 0, & 0, & 0, & 0, & 0, & 0, & 1, & 1, & 1, & \ldots & \ \\
  1 & \ & 0, & 0, & 0, & 0, & 0, & 0, & 0, & 0, & 0, & 0, & 0, & 1, & 1, & \ldots & \ \\
  1 & \ & 0, & 0, & 0, & 0, & 0, & 0, & 0, & 0, & 0, & 0, & 0, & 0, & 1, & \ldots & \ \\
  \underbrace{\vdots} & \ & \ & \ & \ & \ & \ & \ & \ & \ & \ & \ & \ & \ & \ & \ & \ \\
  \hline
  = & \{ & 1, & 2, & 0, & 2, & 1, & 1, & 2, & 1, & 2, & 2, & 2, & 0, & 1, & \ldots & \}
\end{array}\]

And the result is $R(16)$, as expected.  It has already been mentioned that $R(16)$ is aperiodic, even though it is the sum of periodic sequences.  Simply adding the sequences mod 3 would result in the sequence $\{1,2,0,1,2,0,\ldots\}$, so in that sense the cumulative carry is most responsible for the aperiodicity of most rotation remainder expansions.

In the next section, we start proving the theorems stated in the introduction that show the connection between rotation remainder sequences and the triangle $T_m$.

\section{Proof of Theorem \ref{m-adic sum}}   To prove Theorem \ref{m-adic sum}, it remains to show that Equation \eqref{main expansion} holds. For this we have to establish recurrence relations for the sequences $\{x_n(x,r)\}$ and $\{r_n(x,r)\}$.  For simplicity, we will drop the arguments $x$ and $r$ and refer to these sequences as $\{x_n\}$ and $\{r_n\}$.  Recall that these sequences track the row and column positions of a number when it is found lying within the first $m$ columns.  So to compute $x_{n+1}$ and $r_{n+1}$ from $x_n$ and $r_n$, we utilize the recursive relations on $T_m(a,b)$ until we get a new column position within the first $m$ column positions, like so:
\begin{align*}
T_m(x_n,r_n) &= T_m(x_n,x_n+r_n)\\
&= T_m(x_n + 1,x_n+r_n - m)\\
&= T_m(x_n + 2,x_n+r_n - 2m)\\
&= \cdots\\
&= T_m\left(x_n + \left\lfloor\frac{x_n+r_n}{m}\right\rfloor,x_n+r_n - \left\lfloor\frac{x_n+r_n}{m}\right\rfloor m\right).
\end{align*}
Following this scheme, we see that \begin{equation*}x_{n+1} = x_n + \left\lfloor\frac{x_n+r_n}{m}\right\rfloor\end{equation*} and \begin{equation*}r_{n+1} = x_n+r_n - \left\lfloor\frac{x_n+r_n}{m}\right\rfloor m.\end{equation*}  This recursive relation can be expressed in the following way:

\begin{equation} \label{recurrence} (m+1)x_n + r_n = mx_{n+1} + r_{n+1}. \end{equation}

The recursively minded person will quickly see this equation being used repeatedly to see that \begin{equation*} (m+1)^j((m+1)x_{n-j} + r_{n-j}) = m^j(mx_{n+1} + r_{n+1}) + \sum_{k=n+1-j}^n (m+1)^{n-k}m^{k-1+j-n}r_k \end{equation*} for any nonnegative integer $j \leq n$.  This equation leads to the following sequence of equivalences, which establishes Theorem \ref{m-adic sum}.

\begin{proposition} \label{mod m^n} For any pair of integers $n,j$ with $n \geq 1$ and $0 \leq j \leq n$, \begin{equation*}(m+1)^{j}((m+1)x_{n-j} + r_{n-j}) \equiv \sum_{k=n-j+1}^n (m+1)^{n-k}m^{k-1+j-n}r_k \bmod{m^{j}}.\end{equation*}  In particular, \begin{equation*}(m+1)^n((m+1)x + r) \equiv \sum_{k=1}^n (m+1)^{n-k}m^{k-1}r_k \bmod{m^n}\end{equation*} for every $n \in \mathbb{N}$. \end{proposition}

Proposition \ref{mod m^n} can be rewritten to give the sequence of equivalences \begin{equation*}(m+1)^{n-j}((m+1)x_{j} + r_{j}) \equiv \sum_{k=j+1}^n (m+1)^{n-k}m^{k-1-j}r_k \bmod{m^{n-j}},\end{equation*} which establishes the equation \begin{equation*} m((m+1)x_j + r_j) = \sum_{k=j+1}^\infty \left(\frac{m}{m+1}\right)^{k-j}r_k \end{equation*} in $\mathbb{Q}_m$, which is perhaps a more general way of stating Theorem \ref{m-adic sum}.  On, the other hand, it also helps prove Corollary \ref{tail story}.  Suppose we have $\bar{x}$ and $\bar{r}$ such that $\{r_n(\bar{x},\bar{r})\}_{n=0}^\infty$ is a tail of $\{r_n(x,r)\}_{n=0}^\infty$, i.e. there exists some $j$ such that $r_n(\bar{x},\bar{r}) = r_{n+j}(x,r)$ for every $n$.  Then by Theorem \ref{m-adic sum} we have
\begin{align*}
m((m+1)\bar{x} + \bar{r}) &= \sum_{k=1}^\infty \left(\frac{m}{m+1}\right)^{k}r_k(\bar{x},\bar{r})\\
&= \sum_{k=1}^\infty \left(\frac{m}{m+1}\right)^{k}r_{k+j}(x,r)\\
&= \sum_{k=j+1}^\infty \left(\frac{m}{m+1}\right)^{k-j}r_{k}(x,r)\\
&= m((m+1)x_j(x,r) + r_j(x,r)).
\end{align*}

Therefore, $(m+1)\bar{x} + \bar{r} = (m+1)x_j(x,r) + r_j(x,r)$.  Rearrange to get $(m+1)(\bar{x} - x_j(x,r)) = r_j(x,r) - \bar{r}$.  Since $-m < r_j(x,r) - \bar{r} < m$, it follows that $\bar{x} - x_j(x,r) = 0$, and so $r_j(x,r) - \bar{r} = 0$ as well.  This proves Corollary \ref{tail story}.

To prove Corollary \ref{equal tails}, observe that if $x, r, \bar{x}$ and $\bar{r}$ are integers such that $\{r_n(x,r)\}_{n=j}^\infty = \{r_n(\bar{x},\bar{r})\}_{n=j}^\infty$ for some positive integer $j$, then Corollary \ref{tail story} implies that $x_j(x,r) = x_j(\bar{x},\bar{r})$ and $r_j(x,r) = r_j(\bar{x},\bar{r})$.  But $x_j$ and $r_j$ can be used via Equation \eqref{recurrence} to determine $x$ and $r$, or $\bar{x}$ and $\bar{r}$.  Thus $x = \bar{x}$ and $r = \bar{r}$, giving us $\{r_n(x,r)\}_{n=0}^\infty = \{r_n(\bar{x},\bar{r})\}_{n=0}^\infty$, as desired.

\section{Aperiodicity of $\{r_n\}$}

The goal of this brief section is to show that for any positive integer $x$ and least residue $r$ modulo $m$, the sequence $\{r_n(x,r)\}$ is aperiodic.  There is more than one possible approach.  The route we will take utilizes the sequence $\{y_n(x,r)\}$, which was introduced as an ``encoding" of $\{x_n(x,r)\}$ and $\{r_n(x,r)\}$ in Proposition \ref{y_n}.  We will show that this sequence is strictly increasing, and then use this to show that the sequence $\{r_n(x,r)\}$ cannot be periodic.
Because the sequence $y_n$ is generated by repeated multiplication by $(m+1)/m$ combined with a flooring function, the result of this section has a similar flavor to \cite{forman1967arithmetic}.

\subsection*{Proof of Proposition \ref{y_n}}

First, let us prove Proposition \ref{y_n}.  Recall how we defined $\{y_n(x,r)\}$:

\begin{itemize}
\item Let $y_0(x,r) = (m+1)x_0(x,r) + r_0(x,r)$
\item Let $y_n(x,r) = \left\lfloor \frac{(m+1)y_{n-1}(x,r)}{m} \right\rfloor$ for each $n \in \mathbb{N}$.
\end{itemize}

We now seek to show that for each non-negative integer $n$,
\begin{equation*}\left\lfloor \frac{y_n(x,r)}{m+1} \right\rfloor = x_n(x,r) \text{ and } y_n(x,r) \equiv r_n(x,r)  \bmod{(m + 1)}.\end{equation*}  Observe that this statement clearly holds when $n = 0$.  Now suppose it holds for an arbitrary $n$.  We will show it holds for $n + 1$.

Suppose $\left\lfloor \frac{y_n}{m+1} \right\rfloor = x_n$ and $y_n \equiv r_n \bmod{(m+1)}$.  Then ${y_n = (m+1)x_n + r_n}$.  By Equation \eqref{recurrence}, $y_n = mx_{n+1} + r_{n+1}$.  Thus $y_n \equiv r_{n+1} \bmod{m}$.  Therefore, \begin{align*}y_{n+1}& = \left\lfloor \frac{(m+1)y_n}{m} \right\rfloor = \frac{(m+1)y_n - r_{n+1}}{m}\\& = \frac{(m+1)(mx_{n+1} + r_{n+1}) - r_{n+1}}{m} = (m+1)x_{n+1} + r_{n+1}.\end{align*}  This implies that $\left\lfloor \frac{y_{n+1}}{m+1} \right\rfloor = x_{n+1}$ and $y_{n+1} \equiv r_{n+1} \bmod{(m+1)}$, so the proposition holds for $n + 1$.  The proposition is thus proved by induction. \qed

\subsection*{Proof of Theorem \ref{aperiodic}}

Now we will show that the sequence $\{y_n(x,r)\}$ is strictly increasing, and this will be enough to prove the aperiodicity of $\{r_n(x,r)\}$.

\begin{lemma} \label{y_n increasing} If $x$ is any positive integer and $r$ is any least residue modulo $m$, then for every nonnegative integer $n$, $y_{n+1}(x,r) > y_n(x,r)$ .\end{lemma}

\begin{proof}  Let $n$ be a nonnegative integer.  Observe that $x_n \geq x_{n-1} \geq \cdots \geq x_0 = x \geq 1$.  Thus $y_n = (m+1)x_n + r_n \geq m + 1$, so $y_{n+1} = \left\lfloor \frac{(m+1)y_n}{m} \right\rfloor = y_n + \left\lfloor \frac{y_n}{m} \right\rfloor \geq y_n + 1$.  Thus $y_{n+1} > y_n$. \end{proof}

Now we offer a proof of Theorem \ref{aperiodic}.

Let $x$ be an integer, and let $r$ be some least residue modulo $m$.  From Theorem \ref{m-adic sum}, we get that $m \cdot y_0(x,r) = \sum_{k=1}^\infty \left(\frac{m}{m+1}\right)^k r_k(x,r)$ in $\mathbb{Q}_m$.  Suppose $\{r_n(x,r)\}$ is periodic, so that there exists a positive integer $p$ such that $r_{n+p}(x,r) = r_n(x,r)$ for every $n$.  Corollary \ref{tail story} implies that $y_0(x_p(x,r),r_p(x,r)) = y_0(x,r)$.  Now, by Proposition \ref{y_n}, \begin{equation*}y_p(x,r) = (m+1)x_p(x,r) + r_p(x,r) = y_0(x_p(x,r),r_p(x,r)) = y_0(x,r),\end{equation*} but by Lemma \ref{y_n increasing}, $y_p(x,r) > y_{p-1}(x,r) > y_{p-2}(x,r) > \cdots > y_0(x,r)$.  This is a contradiction.  Therefore, $\{r_n(x,r)\}$ is aperiodic.  \qed

Note that the same argument can be used to show that $\{r_n(x,r)\}$ has no periodic tail.

\section{Other sequences and arrays}

Up until now, we have only been proving propositions about the location of an individual in the triangle, i.e. the position of a person standing in a rotating line (to which more people are continually appended).  We still need to find ways to calculate other quantities of interest, namely, \begin{itemize}
\item the frequency with which 1 is at the head of a row,
\item the frequency with which a new number appears in the array, and
\item the frequency with which a certain number appears in a row. \end{itemize} To do this, we introduce the following notation:

\begin{definition} We define $\{h_m(x): x \in \mathbb{N}\}$ to be the sequence of heads of rows in the triangle.  We define $\{l_m(n): n \in \mathbb{N}\}$ to be the sequence of rows in which the 1 leads.  Finally, we define $\{a_m(n) : n \in \mathbb{N}\}$ to be the sequence of rows in which a new number is added to the triangle.  More compactly,
\begin{itemize}
\item $h_m(x) = T_m(x,0)$,
\item $l_m(n) = \min \{x > l_m(n-1) : h_m(x) = 1\}$ (with $l_m(1) = 1$), and
\item $a_m(n) = \min \{x \in \mathbb{N} : n = T_m(x,x-1)\}$.
\end{itemize}
\end{definition}

For the case where $m = 1$, these sequences are all easily recognizable, with the possible exception of $\{h_1(x)\}$, which is known as the bit-counting sequence (Sloane's \href{http://www.research.att.com/~njas/sequences/A000120}{A000120}).  Here are the first few rows of $T_1(x,r)$:

\[\begin{array}{cccccccccc}
1 & & & & & & & & &\\
1 & 2 & & & & & & & &\\
2 & 1 & 2 & & & & & & &\\
1 & 2 & 2 & 3 & & & & & &\\
2 & 2 & 3 & 1 & 2 & & & & &\\
2 & 3 & 1 & 2 & 2 & 3 & & & &\\
3 & 1 & 2 & 2 & 3 & 2 & 3 & & &\\
1 & 2 & 2 & 3 & 2 & 3 & 3 & 4 & &\\
2 & 2 & 3 & 2 & 3 & 3 & 4 & 1 & 2 &\\
2 & 3 & 2 & 3 & 3 & 4 & 1 & 2 & 2 & 3
\end{array}\]

\begin{proposition} \label{binary strings} $h_1(x)$ is the number of 1's in the binary representation of $x$. \end{proposition}

\begin{proof} To prove this by strong induction, note that the base case $x = 1$ is trivial.  Now suppose we have an even row number $x = 2y$.  Observe that 
\begin{equation*}
h_1(y) = T_1(y,0) = T_1(y+1,y-1) = T_1(y+2,y-2) = \cdots = T_1(2y - 1, 1) = T_1(2y, 0) = h_1(2y).
\end{equation*}
Assuming $h_1(y)$ is the number of 1's in the binary representation of $y$, it will also be the number of 1's in the binary representation of $2y$, since the binary representation of $2y$ is the same as that of $y$, but with a zero appended.  Thus if the row number $x$ is even, then $h_1(x)$ is the number of 1's in the binary representation of $x$.

Now suppose we have an odd row number $x = 2y + 1$.  Observe that \begin{equation*}T_1(y+1,y) = T_1(y+2,y-1) = \cdots = T_1(2y,1) = T_1(2y+1,0) = h_1(2y+1).\end{equation*}  Thus $h_1(2y+1) = T_1(y+1,y) = 1 + h_1(y)$.  Note that $h_1(y) = h_1(2y)$ by our previous work.  Assuming $h_1(2y)$ is the number of 1's in the binary representation of $2y$, then $h_1(2y+1) = 1 + h_1(2y)$ is the number of 1's in the binary representation of $2y + 1$, since its binary representation is the same as that of $2y$ with the last digit changed from a 0 to a 1.  Thus if the row number $x$ is odd, then $h_1(x)$ is the number of 1's in the binary representation of $x$.

By the principle of mathematical induction, strong form, the proposition holds for all row numbers. \end{proof}

\begin{proposition} \label{leading 1's m = 1} $l_1(n) = a_1(n) = 2^{n-1}$. \end{proposition}

\begin{proof}  Clearly $l_1(1) = a_1(1) = 1$, so the proposition holds for $n = 1$.  Now suppose the proposition holds for $n = k$.  Then $l_1(k) = a_1(k) = 2^{k-1}$.  Therefore, $T_1(2^{k-1},0) = 1$ and $T_1(2^{k-1},2^{k-1}-1) = k$.  Observe that \begin{equation*} 1 = T_1(2^{k-1},0) = T_1(2^{k-1}+1,2^{k-1}-1) = T_1(2^{k-1}+2,2^{k-1}-2) = \cdots = T_1(2^k, 0) \end{equation*} and \begin{equation*} k = T_1(2^{k-1},2^{k-1}-1) = T_1(2^{k-1}+1,2^{k-1}-2) = \cdots T_1(2^k-1,0),\end{equation*} which also implies that $T_1(2^k,2^k - 1) = k + 1$.  Also, it is clear that through these calculations we have found the nearest instances of 1 and $k$ at the head of a row beyond row $2^{k-1}$.  Thus $l_1(k+1) = a_1(k+1) = 2^k$.  By induction, the proposition holds for every positive integer $n$. \end{proof}

To finish off what we know in the case $m = 1$, we introduce two more arrays.

\begin{definition} Let $F_m(n,k)$ denote the number of $k$'s in row $n$, where $m$ is the rotation number.  Then the function $F_m(n,k)$ gives a \emph{frequency triangle} that displays the number of each kind of element as it appears. \end{definition}

As an example, here are the first few rows of $F_1(n,k)$:

\[\begin{array}{cccccccccc}
1 & & & & & & & & &\\
1 & 1 & & & & & & & &\\
1 & 2 & & & & & & & &\\
1 & 2 & 1 & & & & & & &\\
1 & 3 & 1 & & & & & & &\\
1 & 3 & 2 & & & & & & &\\
1 & 3 & 3 & 1 & & & & & &\\
1 & 4 & 3 & 1 & & & & & &\\
1 & 4 & 4 & 1 & & & & & &\\
1 & 4 & 5 & 1 & & & & & &\\
\end{array}\]

\begin{definition} Let $f_m(n,k)$ denote the number of $k$'s in the row when $n$ first appears in the triangle, where $m$ is the rotation number.  Then the function $f_m(n,k)$ gives a \emph{reduced frequency triangle} that displays the number of each kind of element as it appears only in those rows that introduce a new element.  That is, $f_m(n,k) = F_m(a_m(n),k)$. \end{definition}

As an example, here are the first few rows of $f_1(n,k)$:

\[\begin{array}{cccccccccc}
1 & & & & & & & & &\\
1 & 1 & & & & & & & &\\
1 & 2 & 1 & & & & & & &\\
1 & 3 & 3 & 1 & & & & & &\\
1 & 4 & 6 & 4 & 1 & & & & &\\
1 & 5 & 10 & 10 & 5 & 1 & & & &\\
1 & 6 & 15 & 20 & 15 & 6 & 1 & & &\\
1 & 7 & 21 & 35 & 35 & 21 & 7 & 1 & &\\
1 & 8 & 28 & 56 & 70 & 56 & 28 & 8 & 1 &\\
1 & 9 & 36 & 84 & 126 & 126 & 84 & 36 & 9 & 1\\
\end{array}\]

This is Pascal's triangle, as is shown by the following proposition.

\begin{proposition} $f_1(n,k) = {n - 1 \choose k - 1}$. \end{proposition}

\begin{proof}  By Proposition \ref{leading 1's m = 1}, $f_1(n,k)$ is the number of times $k$ appears in row $2^{n-1}$.  Note that the elements in row $2^{n-1}$ are the same as the leading elements of rows $2^{n-1}$ through $2^n - 1$; that is, $T_1(2^{n-1},r) = h(2^{n-1} + r)$ (one can simply trace diagonally to see this).  By Proposition \ref{binary strings}, $T_1(2^{n-1},r)$, $0 \leq r \leq 2^{n-1} - 1$, represents the number of 1's in the binary representation of $2^{n-1} + r$.  These $2^{n-1}$ binary representations constitute all the $n$-digit binary strings that begin with a 1.  Thus we see that $f_1(n,k)$ is the number of ways to choose $k - 1$ of the remaining $n - 1$ digits to be 1's, which is given by ${n - 1 \choose k - 1}$. \end{proof}

This result means that for $m = 1$, the reduced frequency triangle behaves very nicely: it is well known that Pascal's triangle is, among other things, unimodal and symmetrical.  In general, this does not hold for reduced frequency triangles.  The reduced frequency triangle for $m = 2$ is asymmetrical, and for $m = 6$ it even fails unimodality fairly quickly.  However, for the $m = 2$ case we will attempt to say something about frequencies.  That is saved for the next section, which will cover lots of concrete results for $m = 2$.

\section{The case $m = 2$}\label{m2case}

\subsection*{Proof of Theorem \ref{m equals 2}}

The first result for the case $m = 2$ is Theorem \ref{m equals 2}, which we will now prove quickly.  Suppose, in order to obtain a contradiction, that a certain element of $T_2$ reaches the first column only finitely many times.  Then there exist $x$ and $r$ such that $\{r_n(x,r)\}$ has a tail containing no 0's.  Since $m = 2$, each term $r_n(x,r)$ must be either 0 or 1.  By Theorem \ref{aperiodic}, there can be no periodic tail in $\{r_n(1,0)\}$, so there can be no tail of all 1's.  This is a contradiction.  Thus every element in $T_2$ reaches the first column infinitely many times. \qed

The following theorem is a corollary of Theorem \ref{m equals 2}.

\begin{theorem} \label{a_2 infinite} The sequence $\{a_2(n):n \in \mathbb{N}\}$ is infinite, i.e. every positive integer appears in the triangle $T_2$. \end{theorem}

\begin{proof} To see that $\{a_2(n)\}$ is infinite, observe that if $n$ appears in the triangle, then by Theorem \ref{m equals 2}, $n$ will also lead some row in the triangle, which will spawn an $n+1$ in the next row.  Since 1 obviously appears in the triangle, it follows by induction that every integer appears in the triangle. \end{proof}

\subsection*{The growth rate of $l_2(n)$}

Theorem \ref{m equals 2} implies that $\{l_2(n)\}$ is an infinite sequence.  Recall that $\{l_1(n)\}$ grows exponentially as $\{2^{n-1}\}$; thus $l(n+1) = 2l(n)$.  We cannot find so simple recurrence relation for $\{l_2(n)\}$.  However, we do find a kind of recurrence if we look at a ``negative binary" representation of integers $n$.

\subsection*{Negative binary expansion}

\begin{definition} \label{negative binary expansion} Let $x$ be an integer.  The negative binary expansion of $x$ is the unique sequence of 0's and 1's $(b_0, b_1, b_2, \ldots, b_n)$, $b_n = 1$, such that $x = b_0(-2)^0 + b_1(-2)^1 + \cdots + b_n(-2)^n$. \end{definition}

The existence and uniqueness of this representation can be proven with no more difficulty than in the more familiar case where the base is positive.

\begin{align*}
43 &= (1111111)_{-2}\\
500 &= (11000110100)_{-2}\\
999 &= (10000111011)_{-2}.
\end{align*}

Curiously, our recurrence relation has to do with the how many 1's are at the tail of such a representation.  It turns out that if $n$ leads row $x$, then the number of 1's in the tail of $x$'s negative binary representation determines how many times $n$ will come back to the second column before it finally comes back to the first column.  Thus if $n$ leads an even row, then the very next time it comes back to the first two columns, it will lead that row.  On the other hand, if $n$ leads row 43, for example, then $n$ will loop back to the second column 7 times before finally coming back to the lead.  On the other hand, if $n$ leads row 999, it only loops back to the second column twice before then advancing to the lead.  It isn't the size of the number so much as it is the size of the tail of 1's in this exotic expansion of the number.

\begin{proposition} \label{s and k} Let $x$ be a positive integer.  Then $x$ may be written in the form $x = 2^{k+1}s + \frac{1 - (-2)^k}{3}$ for some nonnegative integer $k$ and integer $s$. \end{proposition}

\begin{proof} Let $(b_0,b_1,\ldots,b_n)$ be the negative binary expansion of $x$.  If $x$ is even, we write $x = 2s$, which is the desired form with $k = 0$.  If $x$ is odd, then $b_0 = 1$, and there must be some $k$ such that either $b_k = 0$ or $k = n + 1$ (i.e $x$ only has $k$ digits in its negative binary expansion).  In either case, there exists some integer $s$ such that $x = 2^{k+1}s + (-2)^0 + (-2)^1 + \cdots + (-2)^{k-1} = 2^{k+1}s + \frac{1 - (-2)^k}{3}$, as desired. \end{proof}

The essence of Proposition \ref{s and k} is that every negative binary representation does, in fact, have a tail of 1's.  How does one find $s$ and $k$, given an integer $x$?  Suppose we write $x = b_0(-2)^0 + b_1(-2)^1 + \cdots + b_n(-2)^n$.  Observe that $b_0$ is determined by whether $x$ is even or odd: if $x$ is even, $b_0 = 0$, and if $x$ is odd, $b_0 = 1$.  So it is easy to determine $b_0$, and we can then extract it from $x$.  Let $x^{(1)} = \frac{x - b_0}{-2} = b_1(-2)^0 + b_2(-2)^1 + \cdots + b_n(-2)^{n-1}$.  Now if $x^{(1)}$ is even, $b_1 = 0$, and if $x^{(1)}$ is odd, then $b_1 = 1$.  Thus $b_1$ is easily determined from $x^{(1)}$.  Now let $x^{(2)} = \frac{x^{(1)} - b_0}{-2}$, and so on in like fashion, until at last we have $x^{(n)} = b_n = 1$.  Note that we did not have to know what $n$ was; we simply repeated the algorithm until we got a 1 in our sequence of $x^{(i)}$'s.

Here is an example.  Let $x = 9$.  Since $x$ is odd, $b_0 = 1$.  Now let $x^{(1)} = \frac{9 - 1}{-2} = -4$.  Since $x^{(1)}$ is even, $b_1 = 0$.  Now let $x^{(2)} = \frac{-4 - 0}{-2} = 2$.  Since $x^{(2)}$ is even, $b_2 = 0$.  Now let $x^{(3)} = \frac{2 - 0}{-2} = -1$.  Since $x^{(3)}$ is odd, $b_3 = 1$.  Now let $x^{(4)} = \frac{-1 - 1}{-2} = 1$.  Since $x^{(4)} = 1$, we finish with $b_4 = 1$.  Putting it all together, we get $9 = 11001_{-2}$.  It is easy to check that $9 = 1(-2)^0 + 0(-2)^1 + 0(-2)^2 + 1(-2)^3 + 1(-2)^4$.  We now demonstrate how to use this representation to get from $l_2(n)$ to $l_2(n+1)$.

\begin{proposition} \label{l_2 recurrence} If $l_2(n) = 2^{k+1}s + \frac{1 - (-2)^k}{3}$ for some nonnegative integer $k$ and integer $s$, then $l_2(n+1) = 3^{k+1}s + \frac{1 - (-3)^k}{2}$. \end{proposition}

\begin{proof} Consider $x_0 = x_0(l_2(n),0) = l_2(n)$ and $r_0 = r_0(l_2(n),0) = 0$  Now observe that 

\[x_1 = \left\lfloor \frac{3x_0 + r_0}{2} \right\rfloor = \left\lfloor \frac{3l_2(n)}{2} \right\rfloor = \left\lfloor \frac{3(2^{k+1})s + 1 - (-2)^k}{2} \right\rfloor = 3(2^k)s + (-2)^{k-1},\]
and $r_1 = 1$ by Equation \eqref{recurrence}.  Now observe that if $x_i = 2j$ is even, then
\[x_{i+1} = \left\lfloor \frac{3x_i + r_i}{2} \right\rfloor = \left\lfloor \frac{3(2j) + r_i}{2} \right\rfloor = \frac{6j}{2} = 3j = (3/2)x_i,\] and $r_{i+1} = r_i$ by Equation \eqref{recurrence}.  Thus $x_2 = (3/2)x_1 = 3^2(2^{k-1})s + (-3)(-2)^{k-2},$ $x_3 = (3/2)^2 x_1 = 3^3(2^{k-2})s + (-3)^2(-2)^{k-3},$ $\ldots,$ $x_{k-1} = (3/2)^{k-2} x_1 = 3^{k-1}(2^2)s + (-3)^{k-2}(-2)$, and $x_k = (3/2)^{k-1} x_1 = 3^k (2)s + (-3)^{k-1}$, and $r_k = r_{k-1} = \cdots = r_2 = r_1 = 1$.

Then we have
\[x_{k+1} = \left\lfloor \frac{3x_k + r_k}{2} \right\rfloor = \left\lfloor \frac{3^{k+1}(2)s - (-3)^{k} + 1}{2} \right\rfloor = 3^{k+1}s + \frac{1 - (-3)^{k}}{2},\]
and $r_{k+1} = 0$ by Equation \eqref{recurrence}.  Thus after returning to the first two columns $k$ times and always hitting the second column, the 1 now returns to the first column on the $(k+1)$th trip.  Therefore, $l_2(n+1) = x_{k+1} = 3^{k+1}s + \frac{1 - (-3)^{k}}{2},$ as desired. \end{proof}

Now we can show some bounds on the growth of $\{l_2(n)\}$.

\begin{lemma} \label{k upper bound} If $j = 2^{k+1} s + \frac{1 - (-2)^k}{3} > 0$ for some integer $s$ and some nonnegative integer $k$, then $k \leq \lceil \log_2(j) \rceil + 1$, and equality holds if and only if $s = 0$.
\end{lemma}

\begin{proof} Suppose $k = 0$.  Then $j = 2s > 0$ and so $j \geq 2$; thus $\log_2(j) \geq 1 > 0$, so clearly $0 < \lceil \log_2(j) \rceil + 1$.  Now suppose $k = 1$.  Then $j = 4s + 1 > 0$ and so $j \geq 1$.  Thus $\log_2(j) \geq 0$, so $1 \leq \lceil \log_2(j) \rceil + 1$, as desired.  Note that $s = 0$ implies $j = 1$, which implies $\log_2(j) = 0$, which yields $\lceil \log_2(j) \rceil + 1 = 0 + 1 = 1$.  Moreover if $\lceil \log_2(j) \rceil + 1 = 1$, then $\log_2(j) \leq 0$, and so $j \leq 1$, which implies $s \leq 0$.  But $4s + 1 > 0$ implies that $s \geq 0$, so clearly $s = 0$.  Thus we have the desired inequality for $k = 0$ and $k = 1$, and as well we have equality if and only if $s = 0$.

For the rest of the proof we suppose that $k \geq 2$.

Suppose first that $k$ is even.  Then $(-2)^k = 2^k$, so $2^{k+1}s + \frac{1 - (-2)^k}{3} > 0$ implies that $2^{k+1}s > \frac{2^k - 1}{3} \geq 0$.  Hence $s > 0$.  Therefore,
\begin{align*}
j \geq 2^{k+1} +  \frac{1 - 2^k}{3} &= 2^{k+1} +  \frac{1 + 2^{k-1}}{3} - 2^{k-1}\\ &= 3(2^{k-1}) +  \frac{1 + 2^{k-1}}{3} > 3(2^{k-1}) > 2^k.
\end{align*}

Hence $k < \log_2(j)$, and thus $k < \lceil \log_2(j) \rceil + 1$, as desired.

Now suppose that $k$ is odd so that $(-2)^k = -2^k$.  Observe that
\begin{align*}
2^{k+1}s + 2^k &> 2^{k+1}s + \frac{2^{k+1}}{3} > 2^{k+1}s + \frac{1 + 2^k}{3} = 2^{k+1}s + \frac{1 - (-2)^k}{3} > 0.
\end{align*}
It follows that $s > -1/2$; hence $s \geq 0$.

Next, note the following:
\begin{align*}
\frac{1 - (-2)^k}{3} &= \frac{1 + 2^k}{3} = \frac{1 + 2^{k-2}}{3} - 2^{k-2} + 2^{k-1} = \frac{1 + 2^{k-2}}{3} + 2^{k-2} > 2^{k-2}.
\end{align*}
It follows from this that $j > 2^{k-2}$, which yields $k - 2 < \log_2(j)$, and thus $k \leq \lceil \log_2(j) \rceil + 1$.

Furthermore, we have
\begin{align*}
\frac{1 - (-2)^k}{3} &= \frac{1 + 2^k}{3} = \frac{1 - 2^{k-1}}{3} + 2^{k-1} \leq 2^{k-1}.
\end{align*}
From this we see that if $s = 0$, then $j \leq 2^{k-1}$.  On the other hand, if $s > 0$ then $j \geq 2^{k+1} + \frac{1 - (-2)^k}{3} > 2^{k+1}$.  It follows that $\log_2(j) \leq k - 1$ if and only if $s = 0$.  Thus $k = \lceil \log_2(j) \rceil + 1$ if and only if $s = 0$, as desired. \end{proof}

\begin{proposition} \label{l_2 upper bound} For large $n$, $l_2(n+1) < \frac{27}{8}l_2(n)^{\log_2(3)}$, and $l_2(n+1) \geq \frac{9}{4}l_2(n)^{\log_2(3)}$ whenever $l_2(n) = \frac{1 - (-2)^{k-1}}{3}$ for some positive integer $k$. \end{proposition}

\begin{proof} Note that $\{l_2(n)\}$ is strictly increasing, so as $n$ gets large, so does $l_2(n)$.  So for large $n$, whenever $l_2(n)$ is written in the form $2^{k+1} s + \frac{1 - (-2)^k}{3}$, we must have that either $s$ or $k$ is large.  This allows us to approximate $l_2(n+1)/l_2(n)$; Proposition \ref{l_2 recurrence} implies that with we have
\begin{align*}
\frac{l_2(n+1)}{l_2(n)} &= \frac{3^{k+1} s + \frac{1 - (-3)^k}{2}}{2^{k+1} s + \frac{1 - (-2)^k}{3}}\\
&= \frac{(3^{k+1}) 6s + 3 + (-3)^{k+1}}{(2^{k+1}) 6s + 2 + (-2)^{k+1}}\\
&\approx \frac{3^{k+1} (6s + (-1)^{k+1})}{2^{k+1} (6s + (-1)^{k+1})} \\
&= \left(\frac{3}{2}\right)^{k+1}.
\end{align*}

By the Lemma \ref{k upper bound}, $k < \log_2(l_2(n)) + 2$, so
\begin{align*}
\frac{l_2(n+1)}{l_2(n)} &< \left(\frac{3}{2}\right)^{\log_2(l_2(n)) + 3}\\
&= \left(\frac{3}{2}\right)^3 \left(2^{\log_2(\frac{3}{2})}\right)^{\log_2(l_2(n))}\\
&= \frac{27}{8} l_2(n)^{\log_2(\frac{3}{2})},
\end{align*}
which means $l_2(n+1) < \frac{27}{8}l_2(n)^{1 + \log_2(3/2)} = \frac{27}{8}l_2(n)^{\log_2(3)}$.  If $s = 0$, then Lemma \ref{k upper bound} implies that $\log_2(l_2(n)) + 1 \leq k$, so by the same reasoning we get $\frac{9}{4} l_2(n)^{\log_2(\frac{3}{2})} \leq \frac{l_2(n+1)}{l_2(n)}$, which gives the desired result. \end{proof}

We can use this upper bound on the ratio between successive terms in $\{l_2(n)\}$ to obtain the following upper bound on all of its terms.  The proposition can be proved easily by induction starting with the fact that $l_2(1) = 1$.

\begin{proposition} \label{super-exponential} $l_2(n) < 8^{\left(\log_2(3)\right)^{n-1}-1}$. \end{proposition}

\begin{remark} This upper bound vastly overshoots the actual sequence, but it is difficult to obtain a better upper bound because of the second inequality given in Proposition \ref{l_2 upper bound}.  The point is this: the ratio $l_2(n+1)/l_2(n)$ is approximated by $(3/2)^{\kappa(n)}$, where $\kappa(n)$ is a positive integer whose behavior is rather unpredictable.  Quite often $\kappa(n) = 1$, in which case $\{l_2(n)\}$ behaves just like an exponential function between $n$ and $n+1$.  However, if $\kappa(n) > 1$, and especially if $\kappa(n)$ is at its maximum possible value, then between $n$ and $n+1$ the sequence $\{l_2(n)\}$ behaves more like the super-exponential function given in Proposition \ref{super-exponential}. \end{remark}

\section{Josephus problem}

The Josephus problem \cite{grahamconcrete,schumer2002josephus} is a famous problem that can be stated this way.  Suppose we have the integers 1 through $x$ ordered clockwise in a circle.  Starting with the number 1, we eliminate the $n$th number as we count clockwise.  We do this repeatedly, continuing to count only the numbers that are left until only one remains.  We call the number that remains $J_n(x)$.

Because $T_m(x,r)$ is determined by a rotate and append algorithm, there is a strong connection between the triangle $T_m$ and the Josephus problem.  It turns out that the position of the 1 in each row of $T_m$ can be used to determine $J_{m+1}(x)$.

\begin{definition} For each positive integer $x$, let $j_m(x)$ denote the column position of the 1 in row $x$ of $T_m$; thus if $T_m(x,r) = 1$ and $r \leq x - 1$, then $j_m(x) = r$. \end{definition}

\begin{definition} For each positive integer $x$, we define playing a \emph{Josephus game} on row $x$ as follows: starting with the number in column $m - 1$ mod $x$, cross out every $(m+1)$th number in the row by counting to the left cyclically (i.e. wrapping around on the right end of the row once the left end has been reached) until only one number is left.  We a define a \emph{Josephus move} as one instance of crossing out in this sequence. \end{definition}

As an example, let's see a Josephus game played on the 10th row of the triangle $T_3$.  We start at the number in column position $m - 1$ mod 10, i.e. column position 2.

\[\begin{array}{llllllllll}
4 & 3 & 2^{(1)} & 3 & 3 & 3 & 3 & 2 & 1 & 4
\end{array}\]

Now we count out every 4th element and eliminate it.

\[\begin{array}{llllllllll}
4^{(3)} & 3^{(2)} & 2^{(1)} & 3 & 3 & 3 & 3 & 2 & 1 & \not{4}^{(4)}\\
4 & 3 & 2 & 3 & 3 & \not{3}^{(4)} & 3^{(3)} & 2^{(2)} & 1^{(1)} & \not{4}\\
4 & \not{3}^{(4)} & 2^{(3)} & 3^{(2)} & 3^{(1)} & \not{3} & 3 & 2 & 1 & \not{4}\\
4^{(1)} & \not{3} & 2 & 3 & 3 & \not{3} & \not{3}^{(4)} & 2^{(3)} & 1^{(2)} & \not{4}\\
\not{4}^{(4)} & \not{3} & 2^{(3)} & 3^{(2)} & 3^{(1)} & \not{3} & \not{3} & 2 & 1 & \not{4}\\
\not{4} & \not{3} & 2 & \not{3}^{(4)} & 3^{(3)} & \not{3} & \not{3} & 2^{(2)} & 1^{(1)} & \not{4}\\
\not{4} & \not{3} & 2^{(1)} & \not{3} & \not{3}^{(4)} & \not{3} & \not{3} & 2^{(3)} & 1^{(2)} & \not{4}\\
\not{4} & \not{3} & \not{2}^{(1)(4)} & \not{3} & \not{3} & \not{3} & \not{3} & 2^{(3)} & 1^{(2)} & \not{4}\\
\not{4} & \not{3} & \not{2} & \not{3} & \not{3} & \not{3} & \not{3} & \not{2}^{(2)(4)} & 1^{(1)(3)} & \not{4}\\
\end{array}\]

And 1 is the winner!  This is not a coincidence.

\begin{proposition} \label{josephus} For every integer $x > 1$, $J_{m+1}(x)$ is the value in the set $\{1,2,\ldots,x\}$ that is equivalent to $m - j_m(x)$ mod $x$.  In other words, 1 is the winner of every Josephus game. \end{proposition}

\begin{proof} We will show that 1 is the winner of any Josephus game in row $x > 2$.  Suppose $x = 2$.  The second row is always a 1 followed by a 2.  The Josephus game will consist of one move starting on either the 1 or the 2.  Since we start on column position $m - 1$ mod 2, which is the same as $m + 1$ mod 2, we start with the 1 if $m + 1$ is even or the 2 if $m + 1$ is odd; in either case, it is clear that we eliminate the 2.  Thus the 1 wins.

\[\begin{array}{ccccc}
\multicolumn{2}{c}{m + 1\text{ even}} & or & \multicolumn{2}{c}{m + 1\text{ odd}}\\
(1) & 2 & \ & 1 & (2)\\
1 & \not{2} & \ & 1 & \not{2}
\end{array}\]

Now suppose the 1 wins in row $x$ for some arbitrary $x > 1$.  We will show that it wins in row $x + 1$.  First, let $r = m - 1$ mod $x + 1$.  The first Josephus move will eliminate the number in column $r - m$ mod $x + 1$ (that's what it means to count to the left).  Thus the first number eliminated is in column $-1$ mod $x + 1$, i.e. column $x$.  So the next Josephus move starts on the number in position $x - 1$.  Now observe that $T(x+1,x-1) = T(x,x+m-1) = T(x,m-1)$.  Thus starting the second Josephus move on the element in position $x - 1$ in row $x + 1$ is just like starting a Josephus game in position $m - 1$ mod $x$ in row $x$, and since the 1 is the winner in row $x$, it must also be the winner in row $x + 1$.  By the principle of mathematical induction, the 1 wins in every row $x > 1$.

To see that $J_{m+1}(x)$ is the value in the set $\{1,2,\ldots,x\}$ that is equivalent to $m - j_m(x)$ mod $x$, just count in the appropriate fashion from position $m - 1$ mod $x$ to position $j_m(x)$. \end{proof}

To clarify the inductive step of the proof, we offer an example.  Here are the 9th and 10th rows of $T_3$.
\[\begin{array}{llllllllll}
3 & 2 & 1 & 4 & 3 & 2 & 3 & 3 & 3 & \\
4 & 3 & 2 & 3 & 3 & 3 & 3 & 2 & 1 & 4
\end{array}\]
We will show that if the 1 wins a Josephus game in row 9, then it also wins in row 10.  Now recall the Josephus game on the 10th row.  Starting with the number in column position 2, we count four positions to the left and eliminate the number in that position:
\[\begin{array}{llllllllll}
4^{(3)} & 3^{(2)} & 2^{(1)} & 3 & 3 & 3 & 3 & 2 & 1 & \not{4}^{(4)}\\
\end{array}\]
So we are left with
\[\begin{array}{lllllllll}
4 & 3 & 2 & 3 & 3 & 3 & 3 & 2 & 1^{(1)}
\end{array}\]
with the starting position in the next Josephus move already marked.  When we view this as a cycle, this is the same as
\[\begin{array}{lllllllll}
3 & 2 & 1^{(1)} & 4 & 3 & 2 & 3 & 3 & 3
\end{array}\]
which is exactly what we would have at the beginning of a Josephus game on row 9.  So if the 1 wins in row 9, then it also wins in row 10, and so on.  By induction, it wins in all rows.

\section{Conclusion}

Let us summarize what we have done in the main body of this paper.  First, we explored the problem of tracking an element of the triangle $T_m$ as it appears in the first $m$ columns.  This involved developing technical background, culminating in a long discussion of \emph{rotation remainder expansions}.  We used this technical background to prove three theorems regarding the row and column tracking sequences $\{x_n(x,r)\}_{n=0}^\infty$ and $\{r_n(x,r)\}_{n=0}^\infty$.  Then we explored the other questions posed in the introduction: the frequency with which the 1 appears in the front of rows in $T_m$, the frequency with which news numbers appear in $T_m$, and the frequency with which a number appears in a given row.  We answered these questions by proving several results for the cases $m = 1$ and $m = 2$; for $m > 2$, the same results are much more difficult to prove.  Finally, we proved a result that connects our questions about $T_m$ with the Josephus problem.

To conclude this paper, let us list some problems that were left unanswered in this paper.

\begin{itemize}

\item Concerning column tracking sequences $\{r_n(x,r)\}_{n=0}^\infty$: we were able to prove that every column tracking sequence is aperiodic, but empirical evidence suggests a stronger result.  We conjecture that, given any column tracking sequence $\{r_n(x,r)\}_{n=0}^\infty$, if $r$ is any least nonnegative residue modulo $m$, then there exists some positive integer $k$ such that $r_k(x,r) = r$.  This is very hard to prove, but it is almost certainly true.

\item Concerning the frequency with which 1 appears at the head of a row: the aperiodicity of column tracking sequences resulted in Theorem \ref{m equals 2}; yet we were unable to prove a similar theorem for $m > 2$ because aperiodicity is not enough to show that $\{r_n(x,r)\}_{n=0}^\infty$ always has a zero in it if $m > 2$.  (This question is clearly related to the previous question).  This question, whether there is always a zero in any column tracking sequence, is equivalent to the question of whether every such sequence has infinitely many zeroes.  But by Proposition \ref{josephus}, this question is equivalent to the question, are there infinitely many integers $x$ such that $J_{m+1}(x) = m$?  This is a hard problem, and it remains unsolved.

\item The frequency with which new numbers appear in the triangle $T_m$ is very much related to the previous problem because a new number appears in $T_m$ only when the current greatest number in $T_m$ reaches the head of a row.  In particular, we could an obtain upper bound on how long it takes for a new number to be generated if we had a more precise answer to how long it takes before a number reaches the head of a row.

\item Computing the frequency with which a number appears in a given row is very difficult even in the case $m = 2$.  It is very unclear how to come up with a general strategy for arbitrary values of $m$.

\end{itemize}

With these questions, we conclude this paper.

\bibliographystyle{acm}
\bibliography{C:/mybib/mybib}

\end{document}